\documentclass{amsart}

\usepackage[pagewise]{lineno}

\usepackage{graphicx}

\usepackage{latexsym, amsfonts, amscd, verbatim, amsmath,enumerate}  
\usepackage{amssymb}
\usepackage{mathtools}

\usepackage{hyperref}
\usepackage{xcolor}

\usepackage[capitalise,nameinlink]{cleveref}
\numberwithin{equation}{section}
\usepackage{enumitem}
\newlist{assumption}{enumerate}{1}
\setlist[assumption]{label=(\textsc{a}\arabic*)}
\crefname{assumptioni}{Assumption}{Assumptions}

\usepackage{bbm}
\usepackage{bm}  

\newcommand{\R}{\mathbb{R}}

\newcommand{\1}{\mathbbm{1}}
\newcommand{\Linop}{\mathbb{L}}

\newcommand{\dist}{\mathrm{dist}}
\newcommand{\norm}[1]{||#1||}

\newcommand{\nuv}{\bm{\nu}}		
\newcommand{\supp}{\mathrm{supp}}
\newcommand{\scalarprod}[2]{\left(#1 , #2\right)}

\DeclareMathOperator{\dom}{\mathrm{dom}}

\DeclareMathOperator{\meas}{\mathrm{meas}}
\newcommand{\dx}{\,\mathrm{d}x}

\newcommand{\Ha}{\mathcal{H}}           

\renewcommand{\epsilon}{\varepsilon}

\usepackage{xcolor}
\usepackage[textsize=footnotesize]{todonotes}

\usepackage{cite}

\date{}

\def\emptyset{\mbox{{\rm \O}}}
\def\bar{\overline}

\newtheorem{theorem}{Theorem}[section]
\newtheorem{lemma}[theorem]{Lemma}

\theoremstyle{definition}
\newtheorem{definition}[theorem]{Definition}

\newtheorem{proposition}[theorem]{Proposition}
\newtheorem{corollary}[theorem]{Corollary}

\theoremstyle{remark}
\newtheorem{remark}[theorem]{Remark}

\numberwithin{equation}{section}

\begin{document}
	
	\title[Optimality conditions in terms of B- generalized differentials...]{Optimality conditions in terms of Bouligand generalized differentials for a nonsmooth semilinear elliptic optimal control problem with distributed and boundary control pointwise constraints}
	



	\author[V. H. Nhu]{Vu Huu Nhu}
	\address{Faculty of Fundamental Sciences\\
		PHENIKAA University\\
		Yen Nghia, Ha Dong, Hanoi 12116, Vietnam}
	\email{nhu.vuhuu@phenikaa-uni.edu.vn}
	\thanks{}
	
	\dedicatory{Dedicated to Professor Arnd R\"{o}sch on the occasion of his sixtieth birthday}
	
	\author[N. H. Son]{Nguyen Hai Son }
	\address{School of Applied Mathematics and 		Informatics\\
		Hanoi University of Science and Technology\\
		No.1 Dai Co Viet, Hanoi, Vietnam
		}
	\email{son.nguyenhai1@hust.edu.vn}
	\thanks{This work is supported by the Vietnam Ministry of Education and Training and Vietnam Institute for Advanced Study in Mathematics under Grant B2022-CTT-05.}

	\date{}
	
	\subjclass[2020]{Primary: 49K20; Secondary: 49J20, 49J52, 35J25}
	
	\keywords{Bouligand generalized differential, Optimality condition, Nonsmooth semilinear elliptic equation, Distributed control, Boundary control, Control pointwise constraint}
	
		\begin{abstract}
		This paper is concerned with an  optimal control problem governed by nonsmooth semilinear elliptic partial differential equations with both distributed and boundary unilateral pointwise control constraints, in which the nonlinear coefficient in the state equation is not differentiable at one point.
		Therefore, the Bouligand subdifferential of this nonsmooth  coefficient in every point consists of one or two elements that will be used to construct the two associated Bouligand generalized derivatives of the control-to-state operator in any admissible control.
		These Bouligand generalized derivatives appear in a novel optimality condition, in which 
		in addition to the existence of the adjoint states and of the nonnegative multipliers associated with the unilateral pointwise constraints as usual, other nonnegative multipliers exist and correspond to the nondifferentiability of the control-to-state mapping.
		This type of optimality conditions shall be applied to 
			an optimal control satisfying the so-called \emph{constraint qualification} 
		to derive a \emph{strong} stationarity, where the sign of the associated adjoint state does not vary on the level set of the corresponding optimal state at the value of nondifferentiability. 
		Finally, this strong stationarity is also shown to be equivalent to the  purely primal optimality condition saying that the directional derivative of the reduced objective functional in admissible directions 
			is nonnegative. 
	\end{abstract}

	\maketitle

\section{Introduction}


The study of optimality conditions for nonsmooth optimal control problems governed by partial differential equations (PDEs for short) and by variational inequalities (abbreviated as VIs) is an active topic of research. 
Better understanding of some kinds of stationarity conditions for these nonsmooth problems is of great value in both theory and application.
For problems with nondifferentiable objective functionals and smooth PDEs, we refer to \cite{Stadler2009,WachsmuthWachsmuth2022,WachsmuthWachsmuth2011,CasasTroltzsch2014,CasasHerzogWachsmuth2012,CasasTroltzsch2020,Casas2012,Kunisch2014}, and the references therein. 
In these papers, the objective functional often contains a term, for example, the $L^1$-norm of controls, which is nonsmooth in the control variable,  while 
the control-to-state operator is continuously differentiable as a mapping of control variable. 
In case of having  the $L^1$-norm of controls, there exists in the optimality system a multiplier that belongs to the subdifferential of the $L^1$-term in the sense of convex analysis; see, e.g. \cite{Ioffe,Bauschke2017,BonnansShapiro2000} for the definition of subdifferential of a convex function.


Regarding optimal control problems subject to nonsmooth PDEs and/or VIs, the corresponding control-to-state mapping is, in general, not differentiable, and is often shown to be directionally differentiable only; see, e.g., \cite{Constantin2018,MeyerSusu2017,ClasonNhuRosch2021,BetzMeyer2015,HerzogMeyerWachsmuth2013,RaulsWachsmuth2020,ReyesMeyer2016} as well as the pioneering work \cite[Chap.~2]{Tiba1990}, and the sources cited inside. 
The lack of the smoothness of the control-to-state mapping is the main difficulty when deriving suitable optimality conditions, and makes the analytical and numerical treatment challenging.
With nonsmooth PDE constrained optimal controls, it was shown in  \cite{MeyerSusu2017,Constantin2018,Tiba1990,Barbu1984,NeittaanmakiTiba1994,ClasonNhuRosch2021} that local optima fulfill the first-order optimality conditions of C-stationarity type involving Clarke's generalized gradient of the nonsmooth ingredient. 
This was achieved via using a regularization and relaxation approach to approximate the original problem by its corresponding regularized smooth problems. 
A limit analysis for vanishing regularization thus yields the C-stationarity conditions for the original nonsmooth problem. 
For the optimal control problems of variational inequality type, we refer to \cite{Barbu1984,Hintermuller2001,HerzogMeyerWachsmuth2013}.
Alternative stationarity concepts such as  
the purely primal condition (which is referred to B-stationarity) 
and strong stationarity have also been introduced as the first-order necessary optimality conditions. 
Concerning the B-stationarity condition, it means that the directional derivative of the reduced objective in any direction belonging to Bouligand's tangent cone to the feasible set at the considered minimizer is nonnegative. 
On the other hand, the strong stationarity condition additionally provides the sign of some multiplier on the nonsmooth set, i.e. the set corresponds to the nondifferentiability of the control-to-state mapping.   
These two types of stationarities are equivalent to each other when some additional certain assumption is fulfilled.
With the  strong stationarity concept 
for problems without pointwise control constraints, 
we refer to \cite{Constantin2018,MeyerSusu2017,ClasonNhuRosch2021,Betz2019} 
for PDE constrained optimal controls and to \cite{HerzogMeyerWachsmuth2013,Mignot1976,MignotFulbert1984,ReyesMeyer2016} for optimal control problems governed by VIs.
	For problems with pointwise control constraints, there are even less contributions addressing the strong stationarity. 
	To the best of our knowledge, the only works in	this ﬁeld are papers \cite{Wachsmuth2014} and \cite{Betz2023}, dealing with optimal control of the obstacle problem and of nonsmooth semilinear elliptic PDEs, respectively.

In this paper, we shall investigate the optimal control problem governed by nonsmooth semilinear elliptic PDEs with  unilateral pointwise control constraints on both distributed and boundary controls. We then provide a novel optimality condition in terms of the so-called \emph{Bouligand generalized differential} of the control-to-state mapping.
Our work is inspired by the one of Christof et al. \cite{Constantin2018}. There, the multiple notions of Bouligand generalized differentials in the infinite-dimensional setting were introduced and these notions generalize the concepts of Bouligand subdifferential of a function between finite-dimensional spaces in the spirit of, e.g., \cite[Def.~2.12]{Outrata1998} or \cite[Sec.~1.3]{KlatteKummer2002}.
For our purpose, we use the notion of \emph{strong-strong} Bouligand generalized differential only. 
Following \cite[Def.~3.1]{Constantin2018}, a {strong-strong} Bouligand generalized derivative of a mapping $S$ at a point $w$ is an accumulation, in the strong operator topology, of a sequence consisting of G\^{a}teaux derivatives $S'(w_k)$ for some $\{w_k\}$ of G\^{a}teaux points -- at which $S$ is G\^{a}teaux differentiable -- converging strongly to $w$.

\medskip
Let us emphasize that our approach is very different from those used in all references listed above, including \cite{Constantin2018} and \cite{Betz2023} even. 
Here, in contrast with the  regularization and relaxation approach,  our analysis comes directly from the definition of the Bouligand generalized differential. 
In this spirit, we construct, for any local minimizer, two minimizing sequences which  comprise  G\^{a}teaux points of the control-to-state mapping, belong to the feasible set of the optimal control problem,  and converge strongly to the mentioned minimizer. 
Furthermore, one of these two sequences lies pointwise to the "left" of the minimizer and the other is in the "right"; see \cref{lem:countable-sets} below.
We then define two corresponding \emph{left} and \emph{right} Bouligand generalized derivatives, as it will be seen in \cref{prop:G-pm-belongto-Bouligand-diff}.
Both of these generalized derivatives will play an important role in establishing the novel optimality condition, as stated in \cref{thm:OCs-multiplier}.
This system of optimality conditions states the existence of  multipliers, where nonnegative functions exist and have their supports lying in the level set of the associated optimal state at nonsmooth value. 
	In addition, two adjoint states are defined by two associated linear PDEs, in which the coefficients are pointwise identified via the Bouligand subdifferential of the nonsmooth ingredient of the state equation.
The novel optimality system will be applied to 
	a local minimizer fulfilling the constraint qualification; see \cite{Betz2023,Wachsmuth2014}, in order to 
derive the corresponding strong stationarity condition, which might be obtained by combining the C-stationarity system together with B-stationarity conditions as in  \cite{Betz2023}.

\medskip 


The paper is organized as follows. 
 In \cref{sec:OCP}, we first state the optimal control problem, then give the fundamental assumptions as well as the notation, and finally present the main results of the paper.
In \cref{sec:control2state-oper}, we first provide some needed properties of the control-to-state mapping, including the Lipschitz continuity, the directional differentiability, and the weak maximum principle. We then prove a strong maximum principle that helps us construct the two suitable G\^{a}teaux sequences and consequently determine the left and right Bouligand generalized derivatives. 
\cref{sec:OC-Bouligand} is devoted to stating some technical lemmas and proofs of main results.
	The paper ends with a conclusion and 
	an appendix that shows the pointwise behavior on a level set of the Laplacian of weak solutions to Poisson's equation.

\section{The problem setting, standing assumptions,  notation, and main results} \label{sec:OCP}

Let $\Omega$ be a bounded domain in two-dimensional space $\mathbb{R}^2$ with a  Lipschitz boundary
$\Gamma$.  
We investigate the following distributed and boundary semilinear elliptic optimal control problem with unilateral pointwise
constraints: find a pair of controls $(u,v) \in L^2(\Omega) \times L^{2}(\Gamma)$ and a corresponding state function $y_{u,v} \in H^1(\Omega) \cap C(\overline\Omega)$, which minimize the objective functional 
\begin{subequations}
	\makeatletter
	\def\@currentlabel{P}
	\makeatother
	\label{eq:P}
	\renewcommand{\theequation}{P.\arabic{equation}}   
	\begin{align}
		J(u,v)&= \frac{1}{2} \norm{y_{u,v} - y_{\Omega}}^2_{L^2(\Omega)} + \frac{\alpha}{2} \norm{y_{u,v} - y_{\Gamma}}^2_{L^2(\Gamma)} + \frac{\kappa_\Omega}{2} \norm{u}^2_{L^2(\Omega)}  + \frac{\kappa_\Gamma}{2} \norm{v}^2_{L^2(\Gamma)}  				\label{eq:objective-func} \\
		\intertext{subject to the state equation}
		&
		\left\{
		\begin{aligned}
			-\Delta y_{u,v} + d(y_{u,v})  = u \quad &{\rm in}\ \Omega\\
			\frac{\partial y_{u,v}}{\partial \nu}  + b(x) y_{u,v} = v \quad &{\rm on}\ \Gamma
		\end{aligned}
		\right.											\label{eq:state}\\
		\intertext{and the unilateral pointwise  constraints}
		& u(x) \leq u_b(x) \quad \text{for a.a. } x \in \Omega, \label{eq:constraint-distributed} \\
		& v(x) \leq v_b(x) \quad \text{for a.a. } x \in\Gamma,					\label{eq:constraint-boundary}
	\end{align}
\end{subequations}
where all data of \eqref{eq:P} satisfy the hypotheses listed on \crefrange{ass:data}{ass:d-func-nonsmooth}  below, which are  assumed to be true throughout the whole paper.

\medskip

\begin{assumption} \label{ass:assumption}

	\item \label{ass:data} $\alpha \geq 0$, $\kappa_\Omega, \kappa_\Gamma >0$,   $y_\Omega, u_b \in L^2(\Omega)$, and $y_\Gamma, v_b \in L^2(\Gamma)$.

	\item \label{ass:ub-vb} Either $u_b \in L^2(\Omega)$ or $u_b = \infty$; similarly either $v_b \in L^2(\Gamma)$ or $v_b = \infty$.

	\item \label{ass:b-func}  Function $b$ belongs to $L^\infty(\Gamma)$ satisfying $b(x) \geq b_0 > 0$ for a.a. $x \in \Gamma$ and for some constant $b_0$.

	\item \label{ass:d-func-nonsmooth} Function $d: \R \to \R$ is continuous and monotonically increasing, but it might not be differentiable at some point $\bar t\in \R$. Moreover, $d$ is defined as follows
	\begin{equation*}
		d(t) = \left\{
		\begin{aligned}
			d_1(t) && \text{if} \quad t \leq \bar t,\\
			d_2(t) && \text{if} \quad t > \bar t,
		\end{aligned}
		\right.
	\end{equation*}
	where $d_1$ and $d_2$ are monotonically increasing and continuously differentiable, and satisfy 
	\begin{equation}
		\label{eq:d12-continuity}
		d_1(\bar t)  = d_2(\bar t).
	\end{equation} 
\end{assumption}

\medskip 
	The function $d$ defined in \cref{ass:d-func-nonsmooth} belongs to a class of \emph{piecewise differentiable functions} or \emph{$PC^1$-functions}. For the definition of $PC^1$-function, we refer, e.g., to \cite[Chap.~4]{Scholtes}.
From \cref{ass:d-func-nonsmooth}, $d$ is differentiable at every point $t \neq \bar t$, but it might be directionally differentiable at $\bar t$ only and its directional derivative is given by
\begin{equation*}
	d'(\bar t; s) = 
	\left\{
	\begin{aligned}
		d_1'(\bar t)s && \text{if } s \leq 0,\\
		d_2'(\bar t)s && \text{if } s > 0.
	\end{aligned}
	\right.
\end{equation*}
Thus, we can express the directional derivative of $d$ as
\begin{multline}
	\label{eq:directional-der-d-func}
	d'(t;s) = \1_{\{ t< \bar t \}}(t) d_1'(t)s + \1_{\{ t > \bar t \}}(t) d_2'(t)s \\
	+ \1_{\{ t = \bar t \}}(t) [ \1_{(-\infty, 0)}(s) d_1'(t)s + \1_{(0, \infty)}(s) d_2'(t)s] \quad \text{for $t, s \in \R.$}
\end{multline}
Hereafter, $\1_M$ denotes the characteristic function of a set $M$, i.e., $\1_M(x) = 1$ if $x \in M$ and $\1_M(x) =0$ if $x \notin M$. 
By $\partial_B d$, we denote the Bouligand subdifferential of $d$, i.e.,
\[
	\partial_B d(t) := \{ \lim\limits_{k \to \infty} d'(\tau_k) \mid \tau_k \to t \quad \text{and} \quad d \, \text{ is differentiable in } \tau_k  \};
\]
see, e.g. \cite[Def.~2.12]{Outrata1998}. Obviously, $\partial_B d$ at every point $t$ consists of one or two elements. Namely, we have
\begin{equation}
	\label{eq:d-Bouligand-subd}
	\partial_B d(t)  = \left \{
	\begin{aligned}
		& \{d'(t) \} && \text{if} \quad t \neq \bar t,\\
		& \{d_1'(\bar t), d_2'(\bar t)\}  && \text{if} \quad t = \bar t.
	\end{aligned}
	\right.
\end{equation}
We will use this formula of $\partial_B d$ to construct two associated Bouligand generalized derivatives of the control-to-state mapping defined in \cref{sec:control2state-oper}; see, the inclusions in \eqref{eq:a-pm-Bouligand-d} and the definitions of operators in \eqref{eq:G-pm-Bouligand-der}--\eqref{eq:G-pm-determining} below.
	Moreover, 
	since $d$ is a $PC^1$-function, its Clarke generalized gradient $\partial_C d$ (see \cite{Clarke1990}) is characterized via 
	\begin{equation} \label{eq:Clark-subdiff-d}
		\partial_C d(t) = 
		\left \{
		\begin{aligned}
			& \{d'(t) \} && \text{if} \quad t \neq \bar t,\\
			& [d_1'(\bar t), d_2'(\bar t)]  && \text{if} \quad t = \bar t;
		\end{aligned}
		\right.
	\end{equation}
	see, e.g. \cite[Prop.~4.3.1]{Scholtes}, \cite[Thm~5.16, Chap.~5, p.~366]{Penot2013}, and \cite[p.~5]{KlatteKummer2002}.


\medskip



\noindent\emph{Notation: }
For a closed convex set $U$ in a Hilbert space $X$, we denote by $N(U; u)$ the normal cone to $U$ at a point $u \in U$, that is,
\[
	N(U; u) := \{ x^* \in X \mid \scalarprod{x^*}{ \tilde{u} - u}_X \leq 0 \quad \forall \tilde{u} \in U  \},
\]
where $\scalarprod{\cdot}{\cdot}_X$ stands for the scalar product of $X$. When $X = L^2(M)$ with a measurable set $M$, we simply write $\scalarprod{\cdot}{\cdot}_M$ instead of $\scalarprod{\cdot}{\cdot}_{L^2(M)}$. For a proper convex functional $F: X \to (-\infty, +\infty]$ defined in a Hilbert space $X$, the symbol $\partial F$ stands for the subdifferential of $F$ in the sense of convex analysis, i.e.,
\[
	\partial F(u) := \{ x^* \in X \mid \scalarprod{x^*}{ \tilde{u} - u}_X \leq F(\tilde{u}) - F(u) \quad \forall \tilde{u} \in X \}.
\]
Given Banach spaces $Z_1$ and $Z_2$, the notation $Z_1 \hookrightarrow Z_2$  means that $Z_1$ is continuously embedded in $Z_2$, and the symbol $\Linop(Z_1,Z_2)$ denotes the space of all bounded continuous linear operators between $Z_1$ and $Z_2$. 
With a given function $y: \hat\Omega\to \R$ defined on a measurable subset $\hat\Omega \subset \R^N$ and a given value $t \in \R$, by $\{y =t \}$ we denote the level set of $y$ associated with $t$, that is, $\{y = t \} := \{x \in \hat \Omega \mid y(x)=t \}$. Analogously, for given functions $y_1,y_2$ and numbers $ t_1, t_2 \in \R$, we set $\{ y_1 = t_1, y_2 = t_2 \} := \{ y_1 = t_1\} \cap \{ y_2 = t_2 \}$. Moreover, by $\{y > t\}$ we denote the set of all points $x \in \hat \Omega$ for which $y(x) > t$ and set 
$\{y_1 > t_2, y_2 > t_2 \} := \{ y_1 > t_1\} \cap \{y_2 > t_2 \}$ and so on. 
When $y$ is continuous, by $\supp(y)$, we denote the support of $y$, i.e., $\supp(y) := \mathrm{cl}(\{x \in \hat \Omega \mid y(x) \neq 0\})$.
We write $\meas_{\R^N}(A)$ to indicate the $N$-dimensional Lebesgue measure of a measurable set $A \subset \R^N$ for an integer $N \geq 1$. 
	The symbol $\Ha^1$ denotes the one-dimensional Hausdorff measure on $\R^2$ that is scaled as in  \cite[Def.~2.1]{Evans1992}.
	For measurable functions $u_1, u_2$ and $v_1, v_2$ defined, respectively, in a bounded domain $\hat \Omega \subset \R^2$ and its Lipschitz boundary $\hat\Gamma$, the symbol $(u_1, v_1) \neq (u_2, v_2)$  means that either
	\[
		\meas_{\R^2}(\{ u_1 = u_2  \}) < \meas_{\R^2}(\hat\Omega) \quad \text{or} \quad \Ha^1(\{ v_1 = v_2 \}) < \Ha^1(\hat\Gamma)
	\]	
	holds.
Finally, $C$ stands for a generic positive constant, which might be different at different places of occurrence. We also write, e.g., $C(\xi)$ or $C_\xi$ for a constant that is dependent only on the parameter $\xi$.


\medskip

	Under \crefrange{ass:data}{ass:d-func-nonsmooth}, the control-to-state mapping $S: L^2(\Omega) \times L^2(\Gamma) \ni (u,v) \mapsto y_{u,v} \in H^1(\Omega) \cap C(\overline\Omega)$, with $y_{u,v}$ uniquely solving the state equation \eqref{eq:state} associated with $u$ and $v$, is well-defined; see, e.g. \cite[Thm.~4.7]{Troltzsch2010}, and is directionally differentiable; see \cref{prop:control-to-state-app} below. Thus, the objective functional $J$ is also directionally differentiable and its directional derivative is given by
		\begin{multline}
			\label{eq:J-direc-der}
			J'(u,v;f,h) = \scalarprod{S(u,v)-y_\Omega}{S'(u,v;f,h)}_\Omega \\
			+ \alpha \scalarprod{S(u,v)-y_\Gamma}{S'(u,v;f,h)}_\Gamma 
			+ \kappa_\Omega \scalarprod{u}{f}_\Omega + \kappa_\Gamma \scalarprod{v}{h}_\Gamma
		\end{multline}
		for all $(u,v), (f,h) \in L^2(\Omega) \times L^2(\Gamma)$.

\medskip 

From now on, let $U_{ad}$ stand for the set of all admissible pairs of controls in $L^2(\Omega) \times L^2(\Gamma)$ of \eqref{eq:P}, that is,
\begin{equation}
	\label{eq:admissible-set}
	U_{ad} := \{  (u,v)  \in L^2(\Omega) \times L^2(\Gamma) \mid u \leq u_b \, \text{a.a. in } \Omega,  v \leq v_b \, \text{a.a. on } \Gamma \}.
\end{equation}
The existence of local minimizers to \eqref{eq:P} is stated in the following proposition and its proof is mainly based on the weak lower semicontinuity and the coercivity of the objective functional as well as the fact that the admissible set $U_{ad}$ is closed and convex in the strong topology in $L^2(\Omega) \times L^2(\Gamma)$. We thus skip the proof here.
\begin{proposition}
	\label{prop:existence-minimizers}
	Problem \eqref{eq:P} admits at least one local minimizer. 
\end{proposition}

	Hereafter, let $(\bar u, \bar v)$ be an arbitrary admissible, but fixed control of \eqref{eq:P}. By $\bar y$, we denote the unique state in $H^1(\Omega) \cap C(\overline\Omega)$ satisfying \eqref{eq:state} corresponding to $(u,v) := (\bar u, \bar v)$, that is, $\bar y := y_{\bar u, \bar v}$.
	We now define the following functions on $\Omega$ as follows:
	\begin{equation}
		\label{eq:a-pm-func}
		\bar a_{-} := \1_{\{ \bar y \neq \bar t \}} d'(\bar y) + \1_{\{ \bar y = \bar t \}} d_1'(\bar y) \quad \text{and} \quad \bar a_{+} := \1_{\{ \bar y \neq \bar t \}} d'(\bar y) + \1_{\{ \bar y = \bar t \}} d_2'(\bar y).
	\end{equation}
	These two functions shall appear in the novel optimality conditions; see \cref{thm:OCs-multiplier} below. 
	We thus determine the mappings
	\begin{equation}
		\label{eq:G-pm-operator}
		\begin{aligned}[t]
			\bar G_{\pm}:  L^2(\Omega) \times L^2(\Gamma) & \to H^1(\Omega) \cap C(\overline\Omega)\\
			 (f,h) & \mapsto z_{\pm},
		\end{aligned}
	\end{equation}
	where $z_{\pm}$ uniquely solve the  linear PDEs:
	\begin{equation*} 
		\left \{
		\begin{aligned}
			-\Delta z_{\pm}+ \bar a_{\pm} z_{\pm} & =f && \text{in} \, \Omega, \\
			\frac{\partial z_{\pm}}{\partial \nuv}  + b(x)z_{\pm}&= h&& \text{on}\, \Gamma.
		\end{aligned}
		\right.
	\end{equation*}	
	In \cref{prop:G-pm-belongto-Bouligand-diff} below (see, also \eqref{eq:G-pm-bar-Guv}), we shall show that these mappings belong to the Bouligand generalized differentials of the control-to-state operator in $(\bar u, \bar v)$.

\medskip

	We now state the main results of the paper. The first one is relevant to $\bar{G}_{\pm}$ (see \eqref{eq:adjoint-multiplier-minus} and \eqref{eq:adjoint-multiplier-plus} below) 
	and presents
	the optimality conditions in terms of multipliers, in which nonnegative multipliers, associated with the nondifferentiability of the control-to-state operator, exist and have their supports lying in $\{\bar y = \bar t \}$  as seen in \eqref{eq:supp-condition-minus}  below.  
	Its proof  is deferred until \cref{sec:OC-Bouligand}.
\begin{theorem}
	\label{thm:OCs-multiplier}
	Let $(\bar u, \bar v)$ be a local minimizer of \eqref{eq:P}
	and $\bar a_{\pm}$ defined  in \eqref{eq:a-pm-func}. 
	Then, there exist an adjoint state $\tilde{p}_{-} \in H^1(\Omega) \cap C(\overline\Omega)$ and multipliers $\zeta_{-, \Omega} \in L^2(\Omega)$, $\zeta_{-, \Gamma} \in L^2(\Gamma)$, 
	$\mu_{-} \in L^2(\Omega)$ 
such that following conditions are fulfilled:
\begin{subequations}
	\label{eq:OCs-multipliers-minus}
	\begin{align}
		& \tilde{p}_{-} + \kappa_\Omega \bar u + \zeta_{-, \Omega} =0, &&  \tilde{p}_{-} + \kappa_\Gamma \bar v + \zeta_{-, \Gamma} =0, \label{eq:stationary-multiplier-uv-minus} \\
		& \zeta_{-, \Omega}  
		\begin{cases}
			=0 & \text{a.a. in } \{ \bar u < u_b \},\\
			\geq 0 & \text{a.a. in } \{ \bar u = u_b \},
		\end{cases} &&
		\zeta_{-, \Gamma}  
		\begin{cases}
			=0 & \text{a.a. on } \{ \bar v < v_b \},\\
			\geq 0 & \text{a.a. on } \{ \bar v = v_b \},
		\end{cases}   \label{eq:complementarity-minus} \\
		\intertext{and}
		& \mu_{-} \geq 0 \quad \text{with }    \supp(\mu_{-}) \subset \{ \bar y = \bar t \}, && \label{eq:supp-condition-minus}
	\end{align}
\end{subequations}
where $\tilde{p}_{-}$ satisfies the  equation
\begin{equation} \label{eq:adjoint-multiplier-minus}
	\left \{
	\begin{aligned}
		-\Delta \tilde{p}_{-}+ \bar a_{-} \tilde{p}_{-} & =\bar y - y_\Omega +  |d_1'(\bar t) - d_2'(\bar t)| \mu_{-} && \text{in} \, \Omega, \\
		\frac{\partial \tilde{p}_{-}}{\partial \nuv}  + b(x)\tilde{p}_{-}&= \alpha(\bar y - y_\Gamma)&& \text{on}\, \Gamma.
	\end{aligned}
	\right.
\end{equation}
	If, in addition, $(\bar u, \bar v) \neq (u_b, v_b)$, 
	then, an adjoint state $\tilde{p}_{+} \in H^1(\Omega) \cap C(\overline\Omega)$ and multipliers $\zeta_{+, \Omega} \in L^2(\Omega)$, $\zeta_{+, \Gamma} \in L^2(\Gamma)$, and $\mu_{+} \in L^2(\Omega)$ exist and satisfy the condition \eqref{eq:OCs-multipliers-minus} with $\tilde{p}_{-}$, $\zeta_{-, \Omega}$, $\mu_{-}$, and $\zeta_{-, \Gamma}$ being replaced by $\tilde{p}_{+}$, $\zeta_{+, \Omega}$, $\mu_{+}$, and $\zeta_{+, \Gamma}$, respectively. Moreover, 
	$\tilde{p}_{+}$ fulfills
	\begin{equation} \label{eq:adjoint-multiplier-plus}
		\left \{
		\begin{aligned}
			-\Delta \tilde{p}_{+}+ \bar a_{+} \tilde{p}_{+} & =\bar y - y_\Omega -  |d_1'(\bar t) - d_2'(\bar t)| \mu_{+} && \text{in} \, \Omega, \\
			\frac{\partial \tilde{p}_{+}}{\partial \nuv}  + b(x)\tilde{p}_{+}&= \alpha(\bar y - y_\Gamma)&& \text{on}\, \Gamma.
		\end{aligned}
		\right.
	\end{equation}
\end{theorem}

\medskip 

In the special case, where $(\bar u, \bar v) = (u_b, v_b)$,  
	we derive the following pointwise inequalities for an adjoint state  and the vanishing of the multiplier corresponding to the nondifferentiability of the control-to-state operator. The associated proof is left to \cref{sec:OC-Bouligand}.
\begin{theorem}
	\label{thm:control-ub-vb}
	Let $(\bar u, \bar v)$ be a local minimizer of \eqref{eq:P} satisfying $(\bar u, \bar v) = (u_b, v_b)$. 
	Then, 
	there exists an adjoint state $\bar{p} \in H^1(\Omega) \cap C(\overline\Omega)$ that fulfills the following conditions
		\begin{equation}
			\label{eq:control-ub-vb-OC}
			\frac{-\bar p}{\kappa_\Omega} \geq u_b \, \text{a.a. in } \Omega, \quad \frac{-\bar p}{\kappa_\Gamma} \geq v_b \, \text{a.a. on } \Gamma,
		\end{equation}	
		and
		\begin{equation} \label{eq:adjoint-state-ubvb}
			\left \{
			\begin{aligned}
				-\Delta \bar p + \bar a_{-} \bar p & =\bar y - y_\Omega  && \text{in} \, \Omega, \\
				\frac{\partial \bar p}{\partial \nuv}  + b(x)\bar p&= \alpha(\bar y - y_\Gamma)&& \text{on}\, \Gamma.
			\end{aligned}
			\right.
		\end{equation}
\end{theorem}


Next, we apply the obtained results presented in \cref{thm:OCs-multiplier} 
in order to derive the \emph{strong stationarity conditions},
	which are equivalent to the purely primal optimality condition, as stated in \cref{thm:strong-B-equivalence} below.

\medskip 
	We need the following so-called {constraint qualification:}
	\begin{equation}
		\label{eq:CQ}
		\meas_{\R^2} \left(\{ \bar y = \bar t \} \cap \overline{ \{  \bar u = u_b \} }  \right) = 0;
	\end{equation}
	see, e.g. \cite{Betz2023,Wachsmuth2014}.

\begin{remark}
	\label{rem:CQ}
		We point out that the constraint qualification \eqref{eq:CQ} is due only to the presence of control constraint \eqref{eq:constraint-distributed}, since the nonsmoothness of \eqref{eq:P} occurs in the domain $\Omega$ only. This constraint qualification is automatically fulfilled when $u_b := + \infty$, i.e., there is no unilateral pointwise constraint \eqref{eq:constraint-distributed}.
\end{remark}

\begin{theorem}[{cf. \cite[Thm.~3.7]{Betz2023}; Strong stationarity}]
	\label{thm:OC-strong}
	Let $(\bar u, \bar v)$ be a local minimizer of \eqref{eq:P}
		such that the constraint qualification \eqref{eq:CQ} is satisfied.
	Then, there exist an adjoint state $\tilde{p} \in H^1(\Omega) \cap C(\overline\Omega)$ and a function   $\tilde{a} \in L^\infty(\Omega)$ such that following conditions are verified:
	\begin{subequations}
		\label{eq:OCs-multipliers-strong}
		\begin{align}
			& 
			\tilde{p} + \kappa_\Omega \bar u 
			+ \zeta_{ \Omega} 
		=0, &&  \tilde{p} + \kappa_\Gamma \bar v 
		+ \zeta_{ \Gamma} 
	=0, \label{eq:stationary-multiplier-uv-strong} \\
	& 
	\zeta_{\Omega}  
	\begin{cases}
		=0 & \text{a.a. in } \{ \bar u < u_b \},\\
		\geq 0 & \text{a.a. in } \{ \bar u = u_b \},
	\end{cases} 
&&
\zeta_{ \Gamma}  
\begin{cases}
	=0 & \text{a.a. on } \{ \bar v < v_b \},\\
	\geq 0 & \text{a.a. on } \{ \bar v = v_b \},
\end{cases}  
\label{eq:complementarity-strong} \\
& \tilde{p}[d_1'(\bar t) - d_2'(\bar t)] \geq 0 \quad \text{a.a. in } \{\bar y = \bar t\}, && \label{eq:OCs-multiplier-sign-condition} 
\intertext{and}
& \tilde{a}(x) \in \partial_C d(\bar y(x)) \quad \text{for a.a. } x \in \Omega, && \label{eq:Clarke-condition}
\end{align}
\end{subequations}
where $\tilde{p}$ satisfies 
\begin{equation} \label{eq:adjoint-multiplier-strong}
\left \{
\begin{aligned}
-\Delta \tilde{p}+ \tilde{a} \tilde{p} & =\bar y - y_\Omega && \text{in} \, \Omega, \\
\frac{\partial \tilde{p}}{\partial \nuv}  + b(x)\tilde{p}&= \alpha(\bar y - y_\Gamma)&& \text{on}\, \Gamma.
\end{aligned}
\right.
\end{equation} 
\end{theorem}
\begin{proof}
We first consider the situation that $(\bar u, \bar v) =(u_b, v_b)$. By  \eqref{eq:CQ}, we have 
\begin{equation}
\label{eq:CQ-bary}
\meas_{\R^2} \left( \{ \bar y = \bar t \} \right) = 0.
\end{equation}
Thanks to \cref{thm:OCs-multiplier}, there exist $\tilde{p}_{-} \in H^1(\Omega) \cap C(\overline\Omega)$,  $\zeta_{-, \Omega} \in L^2(\Omega)$, $\zeta_{-, \Gamma} \in L^2(\Gamma)$, and $\mu_{-} \in L^2(\Omega)$ that satisfy \eqref{eq:OCs-multipliers-minus} and \eqref{eq:adjoint-multiplier-minus}.
From \eqref{eq:CQ-bary} and \eqref{eq:supp-condition-minus}, one has $\mu_{-} =0$.
By setting $\tilde{p} := \tilde{p}_{-}$, $\zeta_\Omega := \zeta_{-, \Omega}$,  $\zeta_{\Gamma} : = \zeta_{-, \Gamma}$, and $\tilde{a} := \bar a_{-}$, we thus derive \eqref{eq:OCs-multipliers-strong} and \eqref{eq:adjoint-multiplier-strong} from \eqref{eq:OCs-multipliers-minus} and \eqref{eq:adjoint-multiplier-minus}, respectively.

\medskip

For the situation $(\bar u, \bar v) \neq (u_b, v_b)$, by applying \cref{thm:OCs-multiplier},	
there are $\tilde{p}_{\pm} \in H^1(\Omega) \cap C(\overline\Omega)$,  $\zeta_{\pm, \Omega} \in L^2(\Omega)$, $\zeta_{\pm, \Gamma} \in L^2(\Gamma)$, and
$\mu_{\pm} \in L^2(\Omega)$ 
that satisfy \eqref{eq:OCs-multipliers-minus}, \eqref{eq:adjoint-multiplier-minus}, \eqref{eq:adjoint-multiplier-plus}, and
\begin{subequations}
	\label{eq:OCs-multipliers-plus}
	\begin{align}
	& \tilde{p}_{+} + \kappa_\Omega \bar u + \zeta_{+, \Omega} =0, &&  \tilde{p}_{+} + \kappa_\Gamma \bar v + \zeta_{+, \Gamma} =0, \label{eq:stationary-multiplier-uv-plus} \\
	& \zeta_{+, \Omega}  
	\begin{cases}
	=0 & \text{a.a. in } \{ \bar u < u_b \},\\
	\geq 0 & \text{a.a. in } \{ \bar u = u_b \},
	\end{cases} &&
	\zeta_{+, \Gamma}  
	\begin{cases}
	=0 & \text{a.a. on } \{ \bar v < v_b \},\\
	\geq 0 & \text{a.a. on } \{ \bar v = v_b \},
	\end{cases}   \label{eq:complementarity-plus} \\
	& \mu_{+} \geq 0 \quad \text{with }    \supp(\mu_{+}) \subset \{ \bar y = \bar t \}. && \label{eq:supp-condition-plus}
\end{align}
\end{subequations}
Combining \eqref{eq:CQ} with \eqref{eq:complementarity-minus} and \eqref{eq:complementarity-plus} yields 
\[
\zeta_{\pm,\Omega} = 0 \quad \text{a.a. in } \{ \bar y = \bar t \}.
\]
This, as well as the first identities in \eqref{eq:stationary-multiplier-uv-minus} and in \eqref{eq:stationary-multiplier-uv-plus}, gives
\[
\tilde{p}_{\pm} + \kappa_\Omega \bar u  =0 \quad \text{a.a. in } \{ \bar y = \bar t \},
\]
which yields 
\begin{equation}
\label{eq:p-pm-onlevelset}
\tilde{p}_{-} = \tilde{p}_{+} 
 = - \kappa_\Omega \bar u
 \quad \text{a.a. in } \{ \bar y = \bar t \}.
\end{equation}	
	On the other hand, by applying \cref{lem:Delta-vanishing-levelset} to the first equation in \eqref{eq:state} and using the fact that $[\bar u - d(\bar y)] \in L^2(\Omega)$, we deduce that $\Delta \bar y = 0$ a.a. in $\{ \bar y = \bar t \}$. 
	Since the first equation in \eqref{eq:state} is fulfilled a.a. in $\Omega$ (see, e.g., \cite[Thm.~8.8]{GilbargTrudinger2021}), there holds
	\[
		\bar u = d(\bar t) \quad \text{a.a. in } \{ \bar y = \bar t \}.
	\]
Combining this with \eqref{eq:p-pm-onlevelset} yields
\[
	\tilde{p}_{-} = \tilde{p}_{+} 
		= - \kappa_\Omega d(\bar t)
	\quad \text{a.a. in } \{ \bar y = \bar t \}.
\]
We now apply \cref{lem:Delta-vanishing-levelset} to \eqref{eq:adjoint-multiplier-minus} to derive $\Delta \tilde{p}_{-} = 0$ a.a. in $\{ \tilde{p}_{-} =  - \kappa_\Omega d(\bar t) \}$ and thus in $\{ \bar y = \bar t \}$.
Since the first equation in \eqref{eq:adjoint-multiplier-minus} is satisfied a.a. in $\Omega$, there holds that
\[
	\bar a_{-} \tilde{p}_{-}  = \bar t - y_\Omega + |d_1'(\bar t) - d_2'(\bar t)| \mu_{-} \quad \text{a.a. in } \{ \bar y = \bar t \}.
\]
Similarly, by applying \cref{lem:Delta-vanishing-levelset} to \eqref{eq:adjoint-multiplier-plus}, one has
\[
	\bar a_{+} \tilde{p}_{+}  = \bar t - y_\Omega - |d_1'(\bar t) - d_2'(\bar t)| \mu_{+} \quad \text{a.a. in } \{ \bar y = \bar t \}.
\]
	Subtracting these identities and using the definition of $\bar a_{\pm}$ in \eqref{eq:a-pm-func}, we therefore derive from \eqref{eq:p-pm-onlevelset} that
\begin{equation}
\label{eq:p-identity-mu}
\1_{\{ \bar y = \bar t \}} \tilde{p}_{-}[d_1'(\bar t) - d_2'(\bar t)] = |d_1'(\bar t) - d_2'(\bar t)|(\mu_{-} + \mu_{+}).
\end{equation}
We now put $\tilde{p} := \tilde{p}_{-}$, $\zeta_\Omega := \zeta_{-, \Omega}$, and $\zeta_{\Gamma} : = \zeta_{-, \Gamma}$. We thus derive the conditions \eqref{eq:stationary-multiplier-uv-strong}--\eqref{eq:OCs-multiplier-sign-condition} from  \eqref{eq:stationary-multiplier-uv-minus}--\eqref{eq:supp-condition-minus} and \eqref{eq:p-identity-mu}.

\medskip

It remains to prove that there exists a function $\tilde{a} \in L^\infty(\Omega)$ satisfying \eqref{eq:Clarke-condition} and \eqref{eq:adjoint-multiplier-strong}. For this purpose, 
we consider the following two cases:

\noindent $\bullet$ \emph{Case 1: $d_1'(\bar t) = d_2'(\bar t)$.}
In this situation, we set $\tilde{a} := d'(\bar y)$ and then deduce the desired conclusion.

\noindent $\bullet$ \emph{Case 2: $d_1'(\bar t) \neq d_2'(\bar t)$.} Without loss of generality, we can assume that
\begin{equation}
\label{eq:deri-d1-less-d2}
d_1'(\bar t) < d_2'(\bar t),
\end{equation}
since the other case is considered analogously. 
We now set
\begin{equation*}
a_0 := 
\left\{
\begin{aligned}
& 0 && \text{a.a. in } \Omega\backslash \{ \bar y = \bar t\}, \\
& d_1'(\bar t) && \text{a.a. in } \{ \bar y = \bar t\} \cap \{ \tilde{p} = 0 \},\\
& d_1'(\bar t) + \frac{[d_1'(\bar t) - d_2'(\bar t)] \mu_{-}}{\tilde{p}} && \text{a.a. in } \{ \bar y = \bar t\} \cap \{ \tilde{p} \neq 0 \},
\end{aligned}
\right.
\end{equation*}
and 
\begin{equation}
\label{eq:a-tilde-def}
\tilde{a} :=  \1_{\{ \bar y \neq \bar t \}}d'(\bar y) + a_0.
\end{equation}
Thanks to \eqref{eq:p-identity-mu}, there holds
\[
\1_{\{ \bar y = \bar t \}} \tilde{p} =-(\mu_{-} + \mu_{+}),
\]
which, in combination with the nonnegativity of $\mu_{\pm}$, implies that
\begin{equation}
	\label{eq:mu-pm-vanish}
	\mu_{-} = \mu_{+} = 0 \quad \text{a.a. in } \{ \bar y = \bar t \} \cap \{ \tilde{p} =0 \}.
\end{equation}
From this and the definition of $a_0$, we obtain
\[
\1_{\{ \bar y = \bar t \}} a_0 = 
\left\{
\begin{aligned}
& d_1'(\bar t) && \text{a.a. in } \{ \bar y = \bar t\} \cap \{ \tilde{p} = 0 \},\\
& d_1'(\bar t) -  \frac{[d_1'(\bar t) - d_2'(\bar t)] \mu_{-}}{\mu_{-} + \mu_{+}} && \text{a.a. in } \{ \bar y = \bar t\} \cap \{ \tilde{p} \neq 0 \}.
\end{aligned}
\right.
\] 
This, together with the nonnegativity of $\mu_{\pm}$, gives 
\[
a_0(x) \in [d_1'(\bar t), d_2'(\bar t)] \quad \text{for a.a. } x \in \{ \bar y = \bar t \}.
\]
By combining this with the definition of $\tilde{a}$ in \eqref{eq:a-tilde-def}
	 and the identity in \eqref{eq:Clark-subdiff-d},
we deduce \eqref{eq:Clarke-condition}.

Finally, for \eqref{eq:adjoint-multiplier-strong}, we use the definition of  $a_0$ and exploit the identities in \eqref{eq:mu-pm-vanish} to have
\begin{align*}
a_0 \tilde{p} & = \1_{\{ \bar y = \bar t \}} d_1'(\bar t) \tilde{p} + [d_1'(\bar t) - d_2'(\bar t)] \mu_{-} = \1_{\{ \bar y = \bar t \}} \bar{a}_{-} \tilde{p} - |d_1'(\bar t) - d_2'(\bar t)| \mu_{-},
\end{align*}
where we have employed the definition of $\bar{a}_{-}$ and the condition \eqref{eq:deri-d1-less-d2} to obtain the last identity. 
From this, the definition of $\bar{a}_{-}$, and \eqref{eq:a-tilde-def}, we have
\begin{equation}
\label{eq:a-tilde-p}
\tilde{a} \tilde{p} = \1_{\{ \bar y \neq \bar t \}}d'(\bar y) \tilde{p} + a_0  \tilde{p} = \bar{a}_{-} \tilde{p} - |d_1'(\bar t) - d_2'(\bar t)| \mu_{-}.
\end{equation}
Moreover, since $\tilde{p} = \tilde{p}_{-}$ satisfies \eqref{eq:adjoint-multiplier-minus}, there holds
\begin{equation*} 
\left \{
\begin{aligned}
-\Delta \tilde{p}+ \bar a_{-} \tilde{p} & =\bar y - y_\Omega +  |d_1'(\bar t) - d_2'(\bar t)| \mu_{-} && \text{in} \, \Omega, \\
\frac{\partial \tilde{p}}{\partial \nuv}  + b(x)\tilde{p}&= \alpha(\bar y - y_\Gamma)&& \text{on}\, \Gamma.
\end{aligned}
\right.
\end{equation*}
Combining this with \eqref{eq:a-tilde-p} gives \eqref{eq:adjoint-multiplier-strong}.
\end{proof}

	In the remainder of this section, we state the equivalence between the strong stationarity and B- stationarity. Its proof shall be presented in \cref{sec:OC-Bouligand}.
	\begin{theorem}[{cf. \cite[Thm.~3.8]{Betz2023}; Equivalence between B- and strong stationarity}]
		\label{thm:strong-B-equivalence}
		The following assertions hold:
		\begin{enumerate}[label=(\alph*)]
			\item \label{item:strong2B} If  there exist an adjoint state $\tilde{p} \in H^1(\Omega) \cap C(\overline\Omega)$ and a function   $\tilde{a} \in L^\infty(\Omega)$ satisfying \eqref{eq:OCs-multipliers-strong} and \eqref{eq:adjoint-multiplier-strong}, then  the following condition is fulfilled	
			\begin{equation}
				\label{eq:B-stationarity}
				J'(\bar u, \bar v; u - \bar u, v - \bar v) \geq 0 \quad \text{for all} \quad (u, v) \in U_{ad};
			\end{equation}
			\item \label{item:B2strong} Conversely, if $(\bar u, \bar v)$ satisfies \eqref{eq:CQ}, then  \eqref{eq:B-stationarity} implies the existence of an adjoint state $\tilde{p} \in H^1(\Omega) \cap C(\overline\Omega)$ and a function   $\tilde{a} \in L^\infty(\Omega)$ satisfying \eqref{eq:OCs-multipliers-strong} and \eqref{eq:adjoint-multiplier-strong}.
		\end{enumerate}
	\end{theorem}

\section{Control-to-state operator and its Bouligand generalized differentials} \label{sec:control2state-oper}

We first deduce from \cite[Thm.~4.7]{Troltzsch2010}  that
\begin{equation}
	\label{eq:apriori-esti-state}
	\norm{S(u,v)}_{H^1(\Omega)} + \norm{S(u,v)}_{C(\overline\Omega)} \leq c_\infty (\norm{u}_{L^2(\Omega)} + \norm{v}_{L^2(\Gamma)})
\end{equation}
for some constant $c_\infty >0$ independent of $u, v, d$, and $b$. Moreover, by \cite[Thm.~3.1]{Casas1993}, $S$ is weakly-to-strongly continuous, i.e., the following implication is verified:
\begin{equation}
	\label{eq:compact-control2state}
	(u_k, v_k) \rightharpoonup (u,v)  \, \text{in } L^2(\Omega) \times L^2(\Gamma)  \implies  S(u_k,v_k) \to S(u,v) \, \text{in }  H^1(\Omega) \cap C(\overline\Omega).
\end{equation}


\subsection{Lipschitz continuity and directional differentiability of the control-to-state mapping, and the maximum principle} \label{sec:control2state-directional-diff}

We start this subsection with the following properties on the continuity and the directional differentiability of the control-to-state operator $S$, as well as the maximum principle. The proofs of these properties are based on a standard argument as in \cite{Nhu2021} (see, also \cite{Constantin2018}), and  thus omitted.

\begin{proposition}[{cf. \cite[Prop.~3.1]{Nhu2021}}] \label{prop:control-to-state-app}
	Under \crefrange{ass:data}{ass:d-func-nonsmooth}, the control-to-state mapping $S: L^2(\Omega) \times L^2(\Gamma) \ni (u,v) \mapsto y_{u,v} \in H^1(\Omega) \cap C(\overline\Omega)$, with $y_{u,v}$ solving the state equation \eqref{eq:state} associated with $u$ and $v$, satisfies the following assertions:
	\begin{enumerate}[label=(\alph*)]
		\item \label{item:S-Lipschitz} $S$ is globally Lipschitz continuous.

		\item \label{item:S-dir-der} For any $(u,v), (f, h) \in L^2(\Omega) \times L^2(\Gamma)$, a $\delta := S'(u,v;f,h)$ exists in $H^1(\Omega) \cap C(\overline\Omega)$  and satisfies
		\begin{equation} \label{eq:weak-strong-Hadamard-app}
			\frac{S((u,v) + t_k (f_k,h_k)) - S(u,v)}{t_k} \to S'(u,v;f,h)
		\end{equation}
		strongly in $H^1(\Omega) \cap C(\overline\Omega)$
		for any  $\{(f_k,h_k)\} \subset L^2(\Omega) \times L^2(\Gamma)$ such that $(f_k,h_k) \rightharpoonup (f,h)$ in $L^2(\Omega) \times L^2(\Gamma)$ and $t_k \to 0^+$ as $k \to \infty$. Moreover, $\delta$ uniquely solves 
		\begin{equation}
			\label{eq:S-dir-der}
			\left \{
			\begin{aligned}
				-\Delta \delta + d'(y_{u,v};\delta) & = f && \text{in} \, \Omega, \\
				\frac{\partial \delta}{\partial \nuv}  + b(x) \delta &= h && \text{on}\, \Gamma
			\end{aligned}
			\right.
		\end{equation}
		with $y_{u,v} := S(u,v)$.		
		Consequently, $S$ is Hadamard directionally differentiable at any $(u,v) \in L^2(\Omega) \times L^2(\Gamma)$ in any direction $(f,h) \in L^2(\Omega) \times L^2(\Gamma)$. 

		Here and in the following, with a little abuse of notation, $d'(y; \delta)$ stands for the superposition operator.

		\item \label{item:S-max-prin} $S'(u,v;\cdot, \cdot)$ fulfills the maximum principle, i.e., 
		\[
			\left\{
			\begin{aligned}
				& f \geq 0 \quad \text{a.a. in } \, \Omega,\\
				& h \geq 0 \quad \text{a.a. on } \, \Gamma
			\end{aligned}
			\right.
			\quad \implies \quad S'(u,v;f,h) \geq 0 \quad \text{a.a. on } \, \Omega.
		\]

		\item \label{item:S-wlsc} $S'(u,v;\cdot, \cdot): L^2(\Omega) \times L^2(\Gamma) \ni (f,h) \mapsto \delta \in H^1(\Omega) \cap C(\overline\Omega)$ with $\delta$ fulfilling \eqref{eq:S-dir-der} is weakly-to-strongly continuous, i.e.,
		\begin{equation*}
			(f_k, h_k) \rightharpoonup (f,h) \quad \text{in } \, L^2(\Omega) \times L^2(\Gamma) \, \implies \, S'(u,v; f_k,h_k) \to S'(u,v;f,h)   
		\end{equation*}
		strongly in $H^1(\Omega) \cap C(\overline\Omega).$
	\end{enumerate} 
\end{proposition}


Below, we have a precise characterization of the G\^{a}teaux differentiability of $S$ at some point $(u,v)$. Its proof is similar to that in  \cite[Cor.~2.3]{Constantin2018} and   thus skipped here.
\begin{proposition}[{cf. \cite[Cor.~2.3]{Constantin2018} and \cite[Cor.~3.1]{Nhu2021}}] 
	\label{pro:S-Gateaux-diff-char}
	The operator
	$S: L^2(\Omega) \times L^2(\Gamma) \to H^1(\Omega) \cap C(\overline\Omega)$ is G\^{a}teaux differentiable in some point $(u,v) \in L^2(\Omega) \times L^2(\Gamma)$ if and only if
	\begin{equation} \label{eq:G-diff-char}
		[d_1'(\bar t) - d_2'(\bar t)]\meas_{\R^2}\left( \{S(u,v)  = \bar t \} \right) = 0.
	\end{equation}
	Moreover, if \eqref{eq:G-diff-char} holds, then 
	\begin{equation}
		\label{eq:d-deri-expression}
		d'(S(u,v)(x)) = \1_{\{S(u,v) < \bar t \}}(x) d_1'(S(u,v)(x)) + \1_{\{S(u,v) > \bar t \}}(x) d_2'(S(u,v)(x))
	\end{equation}
	for a.a. $x \in \Omega$, and $z := S'(u,v)(f,h)$ with $(f,h) \in L^2(\Omega) \times L^2(\Gamma)$ uniquely solves the following equation
	\begin{equation}
		\label{eq:S-G-der}
		\left \{
		\begin{aligned}
			-\Delta z + d'(S(u,v)) z & = f && \text{in} \, \Omega, \\
			\frac{\partial z}{\partial \nuv}  + b(x) z &= h && \text{on}\, \Gamma.
		\end{aligned}
		\right.
	\end{equation}
\end{proposition}


We finish this subsection with the strong maximum principle in the interior of the domain $\Omega$, which plays an important role in showing the existence of the G\^{a}teaux differentiability points of $S$ investigated in \cref{sec:Bouligand-diff} below.
\begin{lemma}
	\label{lem:strong-maximum-prin}
	Let $(u_1,v_1)$ and $(u_2,v_2)$ belong to $L^2(\Omega) \times L^2(\Gamma)$ such that 
	\begin{equation}
		\label{eq:strong-maximum-condition}
		\left\{
		\begin{aligned}
			& u_1 \geq u_2 \quad \text{a.a. in } \Omega, \\
			& v_1 \geq v_2 \quad \text{a.a. on } \Gamma,\\
			& (u_1, v_1) \neq (u_2, v_2). 
		\end{aligned}
		\right.
	\end{equation}
	Then there holds
	\begin{equation*}
		S(u_1,v_1)(x) > S(u_2,v_2)(x) \quad \text{for all} \quad x \in \Omega.
	\end{equation*}
\end{lemma}
\begin{proof}
	Setting $y_i := S(u_i,v_i)$, $i=1,2$ and $z := y_1 -y_2$, and then subtracting the equations for $y_i$, one has
	\begin{equation}
		\label{eq:z}
		\left \{
		\begin{aligned}
			-\Delta z + a(x)z & = u_1 -u_2 && \text{in} \, \Omega, \\
			\frac{\partial z}{\partial \nuv}  + b(x) z &= v_1 -v_2 && \text{on}\, \Gamma,
		\end{aligned}
		\right.
	\end{equation}
	where, for a.a. $x \in \Omega$,
	\[
		a(x) := \left\{
		\begin{aligned}
			&  \frac{d(y_1(x)) - d(y_2(x))}{y_1(x) - y_2(x)} && \text{if } y_1(x) \neq y_2(x),\\
			& 0 && \text{if } y_1(x) = y_2(x).
		\end{aligned}
		\right.
	\]
	From the increasing monotonicity of $d$, the continuous differentiability of $d_i$, $i=1,2$, as well as the fact that $y_1, y_2 \in C(\overline\Omega)$, there holds 
	\begin{equation}
		\label{eq:a-positive}
		0 \leq a(x) \leq C \quad \text{for all } x \in \overline \Omega
	\end{equation}
	and for some constant $C>0$ depending on $y_1, y_2$, and $d$. Testing now the equation \eqref{eq:z} for $z$ by $z^{-} := \max\{-z,0\} \in H^1(\Omega) \cap  C(\overline\Omega)$ yields
	\[
		- \int_\Omega  |\nabla z^{-}|^2 + a |z^{-}|^2 dx - \int_\Gamma  b |z^{-}|^2 d\sigma(x) = \int_\Omega (u_1 - u_2) z^{-} dx +  \int_\Gamma (v_1 -v_2) z^{-} d\sigma(x).
	\]
	Combining this with \ref{ass:b-func} and with \eqref{eq:a-positive} gives
	\[
		0 \geq - \int_\Omega  |\nabla z^{-}|^2  dx - b_0 \int_\Gamma  |z^{-}|^2 d\sigma(x) \geq \int_\Omega (u_1 - u_2) z^{-} dx +  \int_\Gamma (v_1 -v_2) z^{-} d\sigma(x) \geq 0,
	\]
	where the last inequality has been derived by using \eqref{eq:strong-maximum-condition}. This implies that $z^{-} =0$ and thus
	$z \geq 0$ a.a. in $\Omega$. Since $z \in C(\overline\Omega)$, one has 
	\begin{equation*}
		z(x) \geq 0 \quad \text{for all} \quad x \in \overline\Omega.
	\end{equation*}	

	It remains to prove that there is no point $x \in \Omega$ such that $z(x) =0$. To this end, arguing now by contradiction, assume that $z(x_0) = 0$ for some point $x_0 \in \Omega$. Then there exists an open ball $B$ containing $x_0$ such that $ \bar B \subset \Omega$ and 
	\[
		0 = \min\limits_{x \in \bar B}z(x) = \inf\limits_{x \in \Omega}z(x).
	\] 
	The strong maximum principle; see, e.g. \cite[Thm.~8.19]{GilbargTrudinger2021}, applied to the equation \eqref{eq:z}, implies that $z$ must be constant in $\Omega$. The continuity of $z$ up to the boundary $\Gamma$ yields that $z(x) = z(x_0) = 0$ for all $x \in \overline{\Omega}$. Then \eqref{eq:z}  gives 
	\[
		u_1 = u_2 \quad \text{in } \Omega \quad \text{and} \quad v_1 - v_2 = 0 \quad \text{on }  \Gamma,
	\]
	which contradicts the last condition in \eqref{eq:strong-maximum-condition}.
	The proof is complete.
\end{proof}


\subsection{Bouligand generalized differentials of the control-to-state operator} \label{sec:Bouligand-diff}

In this subsection, we will formally provide the rigorous definition of the Bouligand generalized differential of the control-to-state operator $S$.  For a mapping acting on the finite dimensional Banach spaces, its Bouligand generalized differential is defined as the set of limits of Jacobians in differentiable points; see, e.g. \cite[Def.~2.12]{Outrata1998}.
However, in infinite dimensions, since the weak and strong topologies underlying these limits are no longer equivalent, there are several ways to extend the definition of the Bouligand generalized differential. For our purpose, we define the \emph{strong-strong} one taken from \cite[Def.~3.1]{Constantin2018}, in which all the limits are understood in the strong topologies. 
For other types of the Bouligand generalized differential, we refer to  \cite[Def.~3.1]{Constantin2018}; see, also, \cite[Def.~2.10]{RaulsWachsmuth2020}.


\begin{definition} \label{def:Bouligand-subdiff}
	We denote by $D_S$ the set of all G\^{a}teaux differentiable points of $S$, i.e.,
	\begin{multline*}
		D_S := \Big\{ (u,v) \in L^2(\Omega) \times L^2(\Gamma) \mid S: L^2(\Omega) \times L^2(\Gamma)\to H^1(\Omega) \cap C(\overline\Omega)  \\ 
	\text{is G\^{a}teaux differentiable in } (u,v) \Big\}.
	\end{multline*}
	The (strong-strong) Bouligand generalized differential of $S$ in $(u,v)$ is defined as
	\begin{multline*}
		\partial_B S(u,v) := \Big\{ G \in \Linop(L^2(\Omega) \times L^2(\Gamma), H^1(\Omega) \cap C(\overline\Omega)) \mid   \\
		\begin{aligned}
			&  \exists \{(u_k, v_k) \} \subset D_S, (u_k, v_k) \to (u,v) \, \text{in }  L^2(\Omega) \times L^2(\Gamma)\\
			&  \text{and } S'(u_k,v_k) \to G \, \text{in the strong operator topology}  \Big\}.
		\end{aligned}
	\end{multline*}
	Here, $S'$ denotes the G\^{a}teaux derivative of $S$.
\end{definition}


\begin{remark} \label{rem:Bouligand-diff}
	We recall that $S'(u_k,v_k) \to G$ in the strong topology if 
	\[
		S'(u_k,v_k)(f,h) \to G(f,h)
	\]
	strongly in $H^1(\Omega) \cap C(\overline\Omega)$ for all $(f,h) \in L^2(\Omega) \times L^2(\Gamma)$; see, e.g. \cite[Chap.~VI]{Dunford1988}.
	The nonemptiness of the Bouligand generalized differential of $S$ in every admissible control $(u,v) \in U_{ad}$ will be shown later in \cref{prop:G-pm-belongto-Bouligand-diff} below.
\end{remark}

We now collect some simple properties of the Bouligand generalized differential of $S$ in the following proposition, whose proof is straightforward and thus skipped. 
\begin{proposition}[{cf. \cite[Lem.~ 3.4 \& Prop.~3.5]{Constantin2018} and \cite[Prop.~2.11 (ii)]{RaulsWachsmuth2020}}]
	\label{prop:Bouligand-simple-properties}
	\begin{enumerate}[label=(\alph*)]
		\item \label{item:Boudligand-Gateaux-diff} 
		If $S$ is G\^{a}teaux differentiable in $(u,v) \in L^2(\Omega) \times L^2(\Gamma)$, then there holds $S'(u,v) \in \partial_B S(u,v)$.
		\item \label{item:Bouligand-bounded}
		There exists a constant $L>0$ such that
		\[
			\norm{G}_{\Linop(L^2(\Omega) \times L^2(\Gamma), H^1(\Omega) \cap C(\overline\Omega))} \leq L
		\]
		for all $G \in \partial_B S(u,v)$ and for all $(u,v) \in L^2(\Omega) \times L^2(\Gamma)$.
		\item Let $(u,v) \in L^2(\Omega) \times L^2(\Gamma)$ and $\{(u_k, v_k)\} \subset L^2(\Omega) \times L^2(\Gamma)$ such that $(u_k, v_k) \to (u,v)$ strongly in $L^2(\Omega) \times L^2(\Gamma)$. Suppose that for any $k \geq 1$ there exists $G_k \in  \partial_B S(u_k,v_k)$ satisfying
		\[
			G_k \to G \quad \text{in the strong operator topology}
		\]
		for some $G \in \Linop(L^2(\Omega) \times L^2(\Gamma), H^1(\Omega) \cap C(\overline\Omega))$. Then $G$ belongs to $\partial_B S(u,v)$.
	\end{enumerate}
\end{proposition}


In the remainder of this subsection, we shall prove the nonemptiness of the Bouligand generalized differential of $S$ in every point in $L^2(\Omega) \times L^2(\Gamma)$. In order to do that, we first validate the following technical lemma that particularly shows the density of the set of all  G\^{a}teaux differentiability points lying in  the admissible set $U_{ad}$, defined in \eqref{eq:admissible-set}. 
	For any $(u, v) \in U_{ad}$ and any $\epsilon >0$, we define the following feasible functions:
	\begin{equation}
		\label{eq:uv-varepsilon}
		\left\{
		\begin{aligned}
			& u^\epsilon_{-} := u -\epsilon \1_\Omega , \quad v^\epsilon_{-} := v -\epsilon \1_\Gamma,  \\
			& u^\epsilon_{+} := \left\{
			\begin{aligned}
				& u + \epsilon(u_b - u)  && \text{if }  u_b \in L^2(\Omega), \\
				& u + \epsilon \1_\Omega  && \text{if } u_b = \infty,
			\end{aligned}
			\right.	 \\
			& 
			v^\epsilon_{+} := \left\{
			\begin{aligned}
				& v + \epsilon(v_b - v)  && \text{if } v_b \in L^2(\Gamma), \\
				& v + \epsilon \1_\Gamma && \text{if } v_b = \infty.
			\end{aligned}
			\right.
		\end{aligned}
		\right.
	\end{equation}
\begin{lemma}
	\label{lem:countable-sets}
	Let $(u,v) \in U_{ad}$ be arbitrary, but fixed
	such that 
	\begin{equation}
		\label{eq:uv-bound-assumption}
		\norm{u}_{L^2(\Omega)} + \norm{v}_{L^2(\Gamma)} \leq L
	\end{equation}
	for some constant $L>0$.
	Then, the following assertions hold:
	\begin{enumerate}[label=(\alph*)]
		\item \label{item:minus-set}
	There is an at most countable set $I_{-}\subset (0,1)$ such that, for any $\epsilon \in (0,1) \backslash I_{-}$, the admissible control $(u^\epsilon_{-}, v^\epsilon_{-}) \in U_{ad}$ satisfies the following conditions: 
	\begin{subequations} \label{eq:vepsilon-minus}
		\begin{align}
			& \meas_{\R^2}(\{S(u^\epsilon_{-},v^\epsilon_{-})= \bar t \}) = 0, \label{eq:vepsilon-G-diff-minus} \\
			& u^\epsilon_{-} \leq u   \leq u_b \quad \text{a.a. in } \Omega,  \label{eq:uepsilon-admissible-minus} \\
			& v^\epsilon_{-} \leq v  \leq v_b \quad \text{a.a. on } \Gamma,  \label{eq:vepsilon-admissible-minus} \\
			& (u^\epsilon_{-}, v^\epsilon_{-}) \neq (u,v), \label{eq:uv-epsilon-neq-uv-minus} \\
			\intertext{and}
			& \norm{u^\epsilon_{-} - u}_{L^2(\Omega)} + \norm{v^\epsilon_{-} - v}_{L^2(\Gamma)} \leq C_{-} \epsilon \label{eq:vepsilon-density-minus}
		\end{align}
	\end{subequations}
	for some positive constant $C_{-} = C_{-}(L)>0$.
	Consequently, the set $D_S \cap U_{ad}$ is dense in $U_{ad}$ in the norm of $L^2(\Omega) \times L^2(\Gamma)$.
	
	\item  \label{item:plus-set}
	If, in addition, 	$(u,v) \neq (u_b,v_b)$, then an at most countable set $I_{+} \subset (0,1)$ exists and satisfies, for any $\epsilon \in (0,1) \backslash I_{+}$, that the admissible control $(u^\epsilon_{+}, v^\epsilon_{+}) \in U_{ad}$ fulfills the following conditions: 
	\begin{subequations} \label{eq:vepsilon-plus}
		\begin{align}
			& \meas_{\R^2}(\{S(u^\epsilon_{+},v^\epsilon_{+})= \bar t \}) = 0, \label{eq:vepsilon-G-diff-plus} \\
			&  u \leq u^\epsilon_{+} \leq u + \epsilon(u_b - u) \leq u_b \quad \text{a.a. in } \Omega,  \label{eq:uepsilon-admissible-plus} \\
			&  v \leq v^\epsilon_{+} \leq v + \epsilon(v_b - v) \leq v_b \quad \text{a.a. on } \Gamma,  \label{eq:vepsilon-admissible-plus} \\
			& (u^\epsilon_{+}, v^\epsilon_{+}) \neq (u,v), \label{eq:uv-epsilon-neq-uv-plus} \\
			\intertext{and}
			& \norm{u^\epsilon_{+} - u}_{L^2(\Omega)} + \norm{v^\epsilon_{+} - v}_{L^2(\Gamma)} \leq C_{+} \epsilon \label{eq:vepsilon-density-plus}
		\end{align}
	\end{subequations}
	for some constant $C_{+} = C_{+}(L)>0$. 
	\end{enumerate}
\end{lemma}
\begin{proof}
	We now divide the proof into two parts as follows.
	
	\noindent \emph{Ad \ref{item:minus-set}.}  
	By virtue of \eqref{eq:uv-varepsilon}, the relations \eqref{eq:uepsilon-admissible-minus}--\eqref{eq:vepsilon-density-minus} are fulfilled for all $\epsilon \in (0,1)$.
	We now show the existence of the set $I_{-} \subset (0,1)$, which is at most countable and fulfills \eqref{eq:vepsilon-G-diff-minus}. To this end, we first consider a family of measurable sets
	\[
		\{\Omega_\epsilon^{-} := \{ y^\epsilon_{-} = \bar t\} \backslash \Gamma\}_{0 < \epsilon < 1} \subset \Omega,
	\]
	where $y^\epsilon_{-} := S(u^\epsilon_{-},v^\epsilon_{-})$. 
	Taking now $0 < \epsilon_1 < \epsilon_2 < 1$ arbitrarily, we deduce from the definition of $u_{-}^{\epsilon}$ and $v_{-}^{\epsilon}$ that
	\begin{equation}
		\label{eq:vepsilon-neq2}
		\left\{
		\begin{aligned}
			& u_{-}^{\epsilon_1} \geq u_{-}^{\epsilon_2} \quad \text{a.a. in } \Omega,\\
			&v_{-}^{\epsilon_1} \geq v_{-}^{\epsilon_2} \quad \text{a.a. on } \Gamma, \\
			& (u_{-}^{\epsilon_1}, v_{-}^{\epsilon_1}) \neq (u_{-}^{\epsilon_2}, v_{-}^{\epsilon_2}), 
		\end{aligned}
		\right.
	\end{equation}
	which, together with \cref{lem:strong-maximum-prin}, yields 
	\[
		y^{\epsilon_1}_{-}(x) > y^{\epsilon_2}_{-}(x) \quad \text{for all} \quad  x \in \Omega.
	\]
	There then holds 
	\begin{equation*}
		\Omega_{\epsilon_1}^{-}  \cap \Omega_{\epsilon_2}^{-}  = \emptyset.
	\end{equation*}	
	Hence, $\{\Omega_\epsilon^{-} \}_{0 < \epsilon < 1}$ forms a family of disjoint measurable sets. From this and \cite[Prop.~B.12]{Leoni2017}, there exists a subset $I_{-} \subset (0,1)$, which is at most countable and satisfies
	\[
		\meas_{\R^2}(\Omega_\epsilon^{-}) = 0 \quad \text{for all} \quad \epsilon \in (0,1) \backslash I_{-}.
	\]
	Therefore, \eqref{eq:vepsilon-G-diff-minus} is fulfilled.

	Moreover, the density of the set $D_S \cap U_{ad}$ in $U_{ad}$ follows from \eqref{eq:vepsilon-G-diff-minus}, \eqref{eq:vepsilon-density-minus}, and   \cref{pro:S-Gateaux-diff-char}.
	\medskip 
	
	\noindent \emph{Ad \ref{item:plus-set}.}   
	Thanks to \eqref{eq:uv-varepsilon}, $(u^\epsilon_{+},v^\epsilon_{+})$ fulfills \eqref{eq:uepsilon-admissible-plus}, \eqref{eq:vepsilon-admissible-plus}, and \eqref{eq:vepsilon-density-plus}.
	Moreover, since $(u,v) \neq (u_b, v_b)$, the relation \eqref{eq:uv-epsilon-neq-uv-plus} is satisfied. 
	From the definition of $u^\epsilon_{+}$ and $v^\epsilon_{+}$ and due to the fact that $(u,v) \neq (u_b, v_b)$, we have
	\begin{equation*}
		\left\{
		\begin{aligned}
			& u_{+}^{\epsilon_1} \leq u_{+}^{\epsilon_2} \quad \text{a.a. in } \Omega,\\
			&v_{+}^{\epsilon_1} \leq v_{+}^{\epsilon_2} \quad \text{a.a. on } \Gamma, \\
			& (u_{+}^{\epsilon_1}, v_{+}^{\epsilon_1}) \neq (u_{+}^{\epsilon_2}, v_{+}^{\epsilon_2}) 
		\end{aligned}
		\right.
	\end{equation*}
	for $0 < \epsilon_1 < \epsilon_2 < 1$, which is analogous to \eqref{eq:vepsilon-neq2}. Therefore, the same argument showing the existence of $I_{-}$ implies that an at most countable set $I_{+}$ exists and satisfies \eqref{eq:vepsilon-G-diff-plus}.
\end{proof}

%


Next, we define the functions that pointwise belong to the Bouligand subdifferential of $d$ introduced in \cref{ass:d-func-nonsmooth}, and then construct the corresponding bounded linear operators, which will be shown to belong to the Bouligand generalized differential of the control-to-state mapping $S$. For any $(u,v) \in L^2(\Omega) \times L^2(\Gamma)$, define two functions $a^{u,v}_{\pm}$ over $\overline\Omega$ as follows:
\begin{align}
	& a^{u,v}_{-}(x) := \1_{\{ y_{u,v} \leq \bar t \}}(x) d_1'(y_{u,v}(x)) + \1_{\{ y_{u,v} > \bar t\}}(x) d_2'(y_{u,v}(x)) \label{eq:a-minus-func} \\
	\intertext{and}
	& a^{u,v}_{+}(x) := \1_{\{ y_{u,v} < \bar t \}}(x) d_1'(y_{u,v}(x)) + \1_{\{ y_{u,v} \geq \bar t\}}(x) d_2'(y_{u,v}(x)) \label{eq:a-plus-func}
\end{align} 
for all $x \in \overline \Omega$ with $y_{u,v} := S(u,v)$. 
	From the definition of $\bar a_{\pm}$ in \eqref{eq:a-pm-func}, there holds
	\begin{equation}
		\label{eq:a-pm-auv-bar}
		\bar a_{\pm} = a^{\bar u,\bar v}_{\pm}.
	\end{equation}
	Besides, we have for all $x \in \overline{\Omega}$ that
	\begin{equation}
		\label{eq:a-pm-identity}
		a^{u,v}_{-}(x) = a^{u,v}_{+}(x) \quad \text{whenever} \quad S(u,v)(x) \neq \bar t.
	\end{equation}
	Moreover, the combination of \eqref{eq:a-minus-func} and \eqref{eq:a-plus-func} with  \eqref{eq:d-Bouligand-subd} yields for all $x \in \overline{\Omega}$ that
	\begin{align*}
		\partial_B d(S(u,v)(x)) & = \left\{
		\begin{aligned}
			&\{ d'(S(u,v)(x)) \} && \text{if } S(u,v)(x) \neq \bar t,\\
			& \{d_1'(S(u,v)(x)), d_2'(S(u,v)(x)) \} && \text{if } S(u,v)(x) = \bar t
		\end{aligned}
		\right. \\
		& = \left\{
		\begin{aligned}
			&\{ d'_1(S(u,v)(x)) \} && \text{if } S(u,v)(x) < \bar t,\\
			&\{ d'_2(S(u,v)(x)) \} && \text{if } S(u,v)(x) > \bar t,\\
			& \{d_1'(S(u,v)(x)), d_2'(S(u,v)(x)) \} && \text{if } S(u,v)(x) = \bar t
		\end{aligned}
		\right. \\
		& = \left\{
		\begin{aligned}
			&\{ a^{u,v}_{-}(x) \} && \text{if } S(u,v)(x) < \bar t,\\
			&\{ a^{u,v}_{+}(x) \} && \text{if } S(u,v)(x) > \bar t,\\
			& \{a^{u,v}_{-}(x),a^{u,v}_{+}(x) \} && \text{if } S(u,v)(x) = \bar t.
		\end{aligned}
		\right. 
	\end{align*}
	Combining this with \eqref{eq:a-pm-identity} gives
\begin{equation}
	\label{eq:a-pm-Bouligand-d}
	 \partial_B d(S(u,v)(x)) = \{a^{u,v}_{-} (x),a^{u,v}_{+} (x) \} \quad \text{for all } x \in \overline \Omega.
\end{equation}
We then define the following operators
\begin{equation}
	\label{eq:G-pm-Bouligand-der}
	\begin{aligned}[t]
			G^{u,v}_{\pm}: L^2(\Omega) \times L^2(\Gamma) &\to H^1(\Omega) \cap C(\overline\Omega) \\
			(f,h) & \mapsto z_{\pm}, 
	\end{aligned}
\end{equation}
where $z_{\pm}$ are the unique solutions in $H^1(\Omega) \cap C(\overline\Omega)$ to linear equations
\begin{equation}
	\label{eq:G-pm-determining}
	\left \{
	\begin{aligned}
		-\Delta z_{\pm} + a^{u,v}_{\pm} z_{\pm} & = f && \text{in} \, \Omega, \\
		\frac{\partial z_{\pm}}{\partial \nuv}  + b(x) z_{\pm} &=h && \text{on}\, \Gamma.
	\end{aligned}
	\right.
\end{equation}
	By \eqref{eq:G-pm-operator}, \eqref{eq:a-pm-auv-bar}, and  \eqref{eq:G-pm-Bouligand-der}--\eqref{eq:G-pm-determining}, there hold
	\begin{equation}
		\label{eq:G-pm-bar-Guv}
		G^{\bar u,\bar v}_{\pm} = \bar G_{\pm}.
	\end{equation}


	Furthermore, we can conclude from the expression of $d'(t;s)$ in \eqref{eq:directional-der-d-func}, the definition of $G^{u,v}_{\pm}$ in \eqref{eq:G-pm-Bouligand-der}--\eqref{eq:G-pm-determining} and  \cref{prop:control-to-state-app}~\ref{item:S-dir-der} \& \ref{item:S-max-prin} that
	\begin{equation}
		\label{eq:S-der-G-identity}
		S'(u, v; -f, - h) =  - G^{u,v}_{-}(f,h) \quad \text{and} \quad S'(u, v; f, h) = G^{u,v}_{+}(f,h)
	\end{equation}
	for all $(u, v), (f, h) \in L^2(\Omega) \times L^2(\Gamma)$ with $f \geq 0$ a.a. in $\Omega$ and $h \geq 0$ a.a. on $\Gamma$.

\medskip 

In \cref{prop:G-pm-belongto-Bouligand-diff} below, we can see that $G^{u,v}_{\pm}$ are elements of the Bouligand generalized differential of $S$ in $(u,v)$.
From this, \eqref{eq:a-pm-Bouligand-d}, and the definition of $G^{u,v}_{\pm}$ in \eqref{eq:G-pm-Bouligand-der}--\eqref{eq:G-pm-determining}, we call $G^{u,v}_{-}$ and $G^{u,v}_{+}$ the \emph{left} and \emph{right} Bouligand generalized derivatives, respectively.

\medskip 

We have the following result on the functions $a^{u,v}_{\pm}$.
\begin{proposition}
	\label{prop:av-pm-limit}
	For any $(u,v) \in U_{ad}$,  there holds
	\begin{equation} \label{eq:av-minus-limit}
		d'(S(u^\epsilon_{-},v^\epsilon_{-})) \to a^{u,v}_{-} \quad \text{a.a. in} \ \Omega \quad \text{as} \quad \epsilon \to 0^+.
	\end{equation}
	Furthermore, if in addition $(u,v) \neq (u_b, v_b)$, then
	\begin{equation} \label{eq:av-plus-limit}
		d'(S(u^\epsilon_{+},v^\epsilon_{+})) \to a^{u,v}_{+} \quad \text{a.a. in} \ \Omega \quad \text{as} \quad \epsilon \to 0^+.
	\end{equation}
	Here $u^\epsilon_{\pm}, v^\epsilon_{\pm}$, $\epsilon \in (0,1) \backslash I_{\pm}$, are determined via \eqref{eq:uv-varepsilon}, and $a^{u,v}_{\pm}$ are given in \eqref{eq:a-minus-func} and \eqref{eq:a-plus-func}. 
\end{proposition}
\begin{proof}
	We only show \eqref{eq:av-plus-limit} since the argument for \eqref{eq:av-minus-limit} is analyzed in a similar way. 
	From \cref{lem:countable-sets}, $(u^\epsilon_{+}, v^\epsilon_{+})$ satisfies \eqref{eq:vepsilon-plus} for all $\epsilon \in (0,1) \backslash I_{+}$. This, in combination with \cref{pro:S-Gateaux-diff-char}, yields  that $y_\epsilon \neq \bar t$ a.a. in $\Omega$ with $y_\epsilon := S(u^\epsilon_{+},v^\epsilon_{+})$  and that
	\[
		d'(S(u^\epsilon_{+}, v^\epsilon_{+})) = \1_{\{y_\epsilon < \bar t \}} d_1'(y_\epsilon) + \1_{\{y_\epsilon > \bar t \}} d_2'(y_\epsilon).
	\]
	Moreover, by setting $y := S(u,v)$ and using \eqref{eq:a-plus-func}, there holds
	\[
		a^{u,v}_{+} =  \1_{\{y < \bar t \}} d_1'(y) + \1_{\{y \geq \bar t \}} d_2'(y) 
	\]
	a.a. in $\Omega$.
	On the other hand, we deduce from \cref{prop:control-to-state-app} \ref{item:S-Lipschitz} and \eqref{eq:vepsilon-density-plus} that
	\begin{equation}
		\label{eq:y-epsilon-y-distance}
		\norm{y_\epsilon - y}_{H^1(\Omega)} + \norm{y_\epsilon - y}_{C(\overline\Omega)} \leq C \epsilon
	\end{equation}
	for some positive constant $C$ independent of $\epsilon$.  Furthermore, by using \eqref{eq:uepsilon-admissible-plus} and \eqref{eq:vepsilon-admissible-plus}, and exploiting \cref{lem:strong-maximum-prin}, one has
	\begin{equation}
		\label{eq:y-epsilon-greater-y}
		y_\epsilon(x) > y(x) \quad \text{for all} \quad x \in \Omega.
	\end{equation}
	We therefore have from \eqref{eq:y-epsilon-y-distance} and \eqref{eq:y-epsilon-greater-y} that
	\begin{multline*}
		\left|\1_{\{y < \bar t \}} d_1'(y) - \1_{\{y_\epsilon < \bar t \}} d_1'(y_\epsilon) \right| \\
		\begin{aligned}
			& = \left| \1_{\{y < \bar t \}}[d_1'(y) - d_1'(y_\epsilon)] + d_1'(y_\epsilon)[\1_{\{y < \bar t \}} -\1_{\{y_\epsilon < \bar t \}}] \right | \\
			& \leq \1_{\{y < \bar t \}}|d_1'(y) - d_1'(y_\epsilon)| + d_1'(y_\epsilon) \1_{\{\bar t - (y_\epsilon - y) \leq y < \bar t\}} \\
			& \leq \1_{\{y < \bar t \}}|d_1'(y) - d_1'(y_\epsilon)| + d_1'(y_\epsilon) \1_{\{\bar t - C \epsilon \leq y < \bar t\}} \\
			& \to 0, 
		\end{aligned}
	\end{multline*}
	a.a. in $\Omega$. 
	Similarly, one has 
	\begin{multline*}
		\left| \1_{\{y \geq \bar t \}} d_2'(y) - \1_{\{y_\epsilon > \bar t \}} d_2'(y_\epsilon) \right| \\
		\begin{aligned}
			& \leq \1_{\{y \geq \bar t \}}|d_2'(y) - d_2'(y_\epsilon)| + d_2'(y_\epsilon)| \1_{\{y \geq \bar t \}} - \1_{\{y_\epsilon > \bar t \}}| \\
			& \leq \1_{\{y \geq \bar t \}}|d_2'(y) - d_2'(y_\epsilon)|  + d_2'(y_\epsilon)|- \1_{\{ \bar t - (y_\epsilon - y) \leq y < \bar t \}} |\\
			& \leq \1_{\{y \geq \bar t \}}|d_2'(y) - d_2'(y_\epsilon)|  + d_2'(y_\epsilon)|- 
			\1_{\{ \bar t - C \epsilon \leq y < \bar  t \}} 
			| \\
			& \to 0
		\end{aligned}
	\end{multline*}
	a.a. in $\Omega$.
	We therefore obtain \eqref{eq:av-plus-limit}.
\end{proof}

The following statement, whose proof is based on \cref{lem:countable-sets} and \cref{prop:av-pm-limit}, shows that both $G^{u,v}_{\pm}$ belong to $\partial_B S(u,v)$ for any admissible control $(u,v)$ and the Bouligand generalized differential of the control-to-state operator $S$ is thus always nonempty.
\begin{proposition}
	\label{prop:G-pm-belongto-Bouligand-diff}
	For any $(u,v) \in U_{ad}$, there hold:
	\begin{enumerate}[label=(\alph*)]
		\item \label{item:G-minus} The following limit is valid 
		\begin{equation*}
			S'(u_{-}^\epsilon, v_{-}^\epsilon) \to G^{u,v}_{-}
		\end{equation*}
		in the strong operator topology of $\Linop(L^2(\Omega) \times L^2(\Gamma), H^1(\Omega) \cap C(\overline\Omega))$,
		and so $G^{u,v}_{-} \in \partial_B S(u,v)$.
		\item \label{item:G-plus} If, in addition, $(u,v) \neq (u_b,v_b)$, then 
		\begin{equation}
			\label{eq:limit-S-der-vepsilon-plus}
			S'(u_{+}^\epsilon, v_{+}^\epsilon) \to G^{u,v}_{+} 
		\end{equation}
		{in the strong operator topology of} $\Linop(L^2(\Omega)\times L^2(\Gamma), H^1(\Omega) \cap C(\overline\Omega))$,
		and thus $G^{u,v}_{+} \in \partial_B S(u,v)$.
	\end{enumerate}
	Consequently, $\partial_B S(u,v) \neq \emptyset$ for all $(u,v) \in U_{ad}$.
	Here $u^\epsilon_{\pm}$ and $v^\epsilon_{\pm}$, $\epsilon \in (0,1) \backslash I_{\pm}$, are determined as in \eqref{eq:uv-varepsilon}.  
\end{proposition}
\begin{proof}
	We prove \ref{item:G-plus} only. The argument showing \ref{item:G-minus} is analyzed similarly.
	It suffices to show \eqref{eq:limit-S-der-vepsilon-plus}.
	For that purpose, we now fix $(f,h) \in L^2(\Omega) \times L^2(\Gamma)$ and put
	$z^\epsilon_{+} := S'(u^\epsilon_{+}, v_{+}^\epsilon)(f,h)$ and $z_{+}:= G^{u,v}_{+}(f,h)$. 
	Subtracting the equations  for $z^\epsilon_{+}$ and $z_{+}$ (see \eqref{eq:S-G-der} and \eqref{eq:G-pm-determining}, respectively) yields
	\begin{equation} \label{eq:omega-epsilon}
		\left \{
		\begin{aligned}
			-\Delta \omega_\epsilon + a^{u,v}_{+} \omega_\epsilon & = g_\epsilon && \text{in} \, \Omega, \\
			\frac{\partial \omega_\epsilon}{\partial \nuv}  + b(x) \omega_\epsilon &= 0&& \text{on}\, \Gamma,
		\end{aligned}
		\right.
	\end{equation}
	where $\omega_\epsilon := z^\epsilon_{+} - z_{+}$ and 
	\[
		g_\epsilon := z^\epsilon_{+}\left[ \left(\1_{\{y < \bar t \}} d_1'(y) - \1_{\{y_\epsilon < \bar t \}} d_1'(y_\epsilon) \right) +  \left(\1_{\{y \geq \bar t \}} d_2'(y) - \1_{\{y_\epsilon > \bar t \}} d_2'(y_\epsilon) \right) \right]
	\]
	with $y := S(u,v)$ and $y_\epsilon := S(u^\epsilon_{+},v^\epsilon_{+})$. Here, in order to express $g_\epsilon$ as above, we have used \eqref{eq:a-plus-func} for $y_{u,v} := y$ and \eqref{eq:d-deri-expression}  for $S(u,v) := y_\epsilon$.
	Combining \cref{prop:av-pm-limit} with the definition of $g_\epsilon$  and with the bound of $\{z^\epsilon_{+}\}$ in $C(\overline\Omega)$, we conclude from Lebesgue's dominated convergence theorem that
	\[
		g_\epsilon \to 0 \quad \text{strongly in} \quad L^p(\Omega)
	\] 
	for any $p>1$. By applying now the Stampacchia Theorem \cite[Thm.~12.4]{Chipot2009} and using the a prior estimate for \eqref{eq:omega-epsilon}, there holds
	\[
		\norm{\omega_\epsilon}_{H^1(\Omega)} + \norm{\omega_\epsilon}_{C(\overline\Omega)} \to 0.
	\]
	This implies \eqref{eq:limit-S-der-vepsilon-plus} since $(f,h)$ is arbitrarily chosen in $L^2(\Omega) \times L^2(\Gamma)$.
\end{proof}


We now address a necessity for an operator in $\Linop(L^2(\Omega) \times L^2(\Gamma), H^1(\Omega) \cap C(\overline\Omega))$ to be an element of the Bouligand generalized differential of $S$.
\begin{proposition}
	\label{prop:Bouligand-necessity}
	Let $(u,v) \in L^2(\Omega) \times L^2(\Gamma)$ be arbitrary, but fixed. Then, for any $G \in \partial_B S(u,v)$, there exists a unique $a_G \in L^\infty(\Omega)$ such that $G$ is identical to solution operator to the following equation
	\begin{equation} \label{eq:Bouligand-element-identity}
		\left \{
		\begin{aligned}
			-\Delta z + a_G z & = f && \text{in} \, \Omega, \\
			\frac{\partial z}{\partial \nuv}  + b(x) z &= h&& \text{on}\, \Gamma,
		\end{aligned}
		\right.
	\end{equation}
	i.e., for any $(f,h) \in L^2(\Omega) \times L^2(\Gamma)$, $z:= G(f,h)$ uniquely solves \eqref{eq:Bouligand-element-identity}. Moreover, there holds
	\begin{equation}
		\label{eq:aG-coefficient}
		a_G(x) \in [a^{u,v}_{-}(x), a^{u,v}_{+}(x)] \quad \text{for a.a. } x \in \Omega.
	\end{equation}
\end{proposition}
\begin{proof}
	Assume now that $G \in \partial_B S(u,v)$. By definition, there exists a sequence $\{(u_k,v_k)\} \subset L^2(\Omega) \times L^2(\Gamma)$ such that  $(u_k,v_k) \in D_S$, $(u_k, v_k) \to (u,v)$ in $L^2(\Omega) \times L^2(\Gamma)$, and
	\begin{equation}
		\label{eq:S-der-vk-to-G}
		S'(u_k,v_k) \to G 
	\end{equation}
	in the strong operator topology in $\Linop(L^2(\Omega) \times L^2(\Gamma), H^1(\Omega) \cap C(\overline\Omega))$.
	
	Fix $(f,h) \in L^2(\Omega) \times L^2(\Gamma)$ and set $z:= G(f,h)$, and $z_k := S'(u_k,v_k)(f,h)$ for all $k \geq 1$.  Owing to \cref{pro:S-Gateaux-diff-char}, $z_k$ satisfies the equation
	\begin{equation}
		\label{eq:zk-pde}
		\left \{
		\begin{aligned}
			-\Delta z_k + d'(y_k)  z_k & = f && \text{in} \, \Omega, \\
			\frac{\partial z_k}{\partial \nuv}  + b(x) z_k &= h && \text{on}\, \Gamma,
		\end{aligned}
		\right.
	\end{equation}
	with $y_k := S(u_k,v_k)$, $k \geq 1$. Since $(u_k, v_k) \to (u,v)$ in $L^2(\Omega) \times L^2(\Gamma)$ and $S$ is weakly-to-strongly continuous as a mapping from $L^2(\Omega) \times L^2(\Gamma)$ to $H^1(\Omega) \cap C(\overline\Omega)$ (see \eqref{eq:compact-control2state}), there holds
	\begin{equation}
		\label{eq:yk-to-y}
		y_k \to y := S(u,v) \quad \text{strongly in} \quad H^1(\Omega) \cap C(\overline\Omega).
	\end{equation}
	From this and the continuity of $d_1'$ and $d_2'$, $\{d'(y_k)\}$ is bounded in $L^\infty(\Omega)$. There then exists a function $a_G \in L^\infty(\Omega)$ such that
	\begin{equation}
		\label{eq:aG-weakstar-limit}
		d'(y_k) \overset{*}{\rightharpoonup} a_G \quad \text{in} \quad L^\infty(\Omega).
	\end{equation}
	Combining this with the limit \eqref{eq:yk-to-y}, we derive from  \cite[Chap.~I, Thm.~3.14]{Tiba1990} that $a_G(x) \in \partial_C d( y(x))$ for a.a. $x \in \Omega$. From this, the definitions of $a^{u,v}_{\pm}$ in \eqref{eq:a-minus-func} and \eqref{eq:a-plus-func}, as well as the identity \eqref{eq:Clark-subdiff-d}, we have \eqref{eq:aG-coefficient}. Moreover, by letting $k \to \infty$ in \eqref{eq:zk-pde} and using the limits \eqref{eq:S-der-vk-to-G} and \eqref{eq:aG-weakstar-limit}, we finally derive \eqref{eq:Bouligand-element-identity}.
\end{proof}


The following result is a direct consequence of \cref{pro:S-Gateaux-diff-char,prop:G-pm-belongto-Bouligand-diff,prop:Bouligand-necessity} in the case where $S$ is G\^{a}teaux differentiable at an element $(u,v)$ in $L^2(\Omega) \times L^2(\Gamma)$.
\begin{corollary}
	\label{cor:Bouligand-Gateaux}
	If $S$ is G\^{a}teaux differentiable at $(u,v) \in L^2(\Omega) \times L^2(\Gamma)$, then 
	\[
		\partial_B S(u,v) = \{S'(u,v) \} = \{G^{u,v}_{-}\} = \{G^{u,v}_{+}\}.
	\]
\end{corollary}
\begin{proof}
	Assume that $S$ is G\^{a}teaux differentiable at $(u,v) \in L^2(\Omega) \times L^2(\Gamma)$, then there obviously holds $\{S'(u,v) \} \subset \partial_B S (u,v) $. Moreover, we have $S(u,v)(x) \neq \bar t$ for a.a. $x \in \Omega$, on account of \cref{pro:S-Gateaux-diff-char}. From this and \eqref{eq:a-pm-identity}, there holds 
	$a^{u,v}_{-}(x) = a^{u,v}_{+}(x)$ for a.a. $x \in \Omega$,
	which, in combination with \cref{prop:G-pm-belongto-Bouligand-diff,prop:Bouligand-necessity}, yields the desired conclusion.	
\end{proof}


We finish this section with explicitly determining the adjoint operator of any Bouligand generalized 
derivative of 
the control-to-state mapping $S$. Let $G \in \partial_B S(u,v)$ be defined as in \eqref{eq:Bouligand-element-identity}. Since $H^1(\Omega) \cap C(\overline\Omega) \hookrightarrow L^2(\Omega)$, $G$ can be considered as a bounded and linear mapping from $L^2(\Omega) \times L^2(\Gamma)$ to $L^2(\Omega)$. Its adjoint operator $G^*$ defined on $L^2(\Omega)$ to $L^2(\Omega) \times L^2(\Gamma)$ is thus determined as follows:
\begin{equation}
	\label{eq:G-adjoint}
	\begin{aligned}[t]
		G^* : L^2(\Omega) & \to L^2(\Omega) \times L^2(\Gamma) \\
		\phi & \mapsto (\zeta, \gamma(\zeta)),
	\end{aligned}
\end{equation}
where $\zeta$ is the unique solution in $H^1(\Omega) \cap C(\overline\Omega)$ to the equation
\begin{equation*} 
	\left \{
	\begin{aligned}
		-\Delta \zeta + a_G \zeta & = \phi && \text{in} \, \Omega, \\
		\frac{\partial \zeta}{\partial \nuv}  + b(x) \zeta &= 0&& \text{on}\, \Gamma
	\end{aligned}
	\right.
\end{equation*}
and $\gamma$ stands for the trace operator.


\section{Proofs of main results} \label{sec:OC-Bouligand}

\subsection{Auxiliary results} \label{sec:auxiliaryresults}

Let us begin this subsection with stating some simple properties of the cost functional introduced in \eqref{eq:objective-func}.
\begin{lemma}
	\label{lem:objective-func-difference-G-der}
	There hold:
	\begin{enumerate}[label=(\alph*)]
		\item \label{item:difference}
		For any $(u_1,v_1), (u_2,v_2) \in L^2(\Omega) \times L^2(\Gamma)$, 
		\begin{multline}
			\label{eq:objective-difference}
			J(u_1,v_1) - J(u_2,v_2) = \frac{1}{2} \norm{y_1 - y_2}_{L^2(\Omega)}^2 + \frac{\alpha}{2}\norm{y_1 - y_2}^2_{L^2(\Gamma)}  \\
			\begin{aligned}[b]
				& + \frac{\kappa_\Omega}{2} \norm{u_1 -u_2}_{L^2(\Omega)}^2 + \frac{\kappa_\Gamma}{2}\norm{v_1 -v_2}_{L^2(\Gamma)}^2 + \scalarprod{y_2 - y_\Omega}{y_1 -y_2}_{\Omega}  \\
				& + \alpha \scalarprod{y_2 - y_\Gamma}{y_1 -y_2}_{\Gamma} + \kappa_\Omega \scalarprod{u_2}{u_1 -u_2}_{\Omega} + \kappa_\Gamma \scalarprod{v_2}{v_1 -v_2}_{\Gamma}
			\end{aligned}
		\end{multline}
		with $y_i :=S(u_i,v_i)$, $i=1,2$.
		
		\item \label{item:obj-G-der}
		If, in addition, $S$ is G\^{a}teaux differentiable at $(u,v)$, then the cost functional $J$ is also G\^{a}teaux differentiable at $(u,v)$ and its derivative is given as
		\begin{equation*}
			J'(u,v)(f,h) = \scalarprod{\varphi+ \kappa_\Omega u}{f}_{\Omega} + \scalarprod{\varphi + \kappa_\Gamma v}{h}_{\Gamma}, \quad (f,h) \in L^2(\Omega) \times L^2(\Gamma),
		\end{equation*}
		where $\varphi \in H^1(\Omega) \cap C(\overline\Omega)$ uniquely solves the following equation
		\begin{equation*} 
			\left \{
			\begin{aligned}
				-\Delta \varphi + d'(S(u,v)) \varphi & = S(u,v) - y_\Omega && \text{in} \, \Omega, \\
				\frac{\partial \varphi}{\partial \nuv}  + b(x) \varphi &=  \alpha (S(u,v) - y_\Gamma)&& \text{on}\, \Gamma.
			\end{aligned}
			\right.
		\end{equation*}
	\end{enumerate}
\end{lemma}
\begin{proof}
	The proof is straightforward and standard.
\end{proof}

\medskip

By applying now \cref{lem:countable-sets} for the situation that $(u,v):= (\bar u, \bar v)$, 
	there then exists a sequence $\{\epsilon_k\} \subset (0,1)$ such that  
	\begin{equation}
		\label{eq:epsilonk-ukvk-minus}
		\left\{
		\begin{aligned}
			& \epsilon_k \to 0^{+} \quad \text{as} \quad k \to \infty, \\
			& (\bar u^{\epsilon_k}_{-},\bar v^{\epsilon_k}_{-}) \in U_{ad} \cap D_S \, \text{satisfies \eqref{eq:vepsilon-minus}}, \\
			& (\bar u^{\epsilon_k}_{+},\bar v^{\epsilon_k}_{+}) \in U_{ad} \cap D_S \, \text{satisfies \eqref{eq:vepsilon-plus}, provided that } (\bar u, \bar v) \neq (u_b, v_b),
		\end{aligned}
		\right.
	\end{equation}
	where $u^{\epsilon_k}_{\pm}$ and $v^{\epsilon_k}_{\pm}$ are defined as in  \eqref{eq:uv-varepsilon}.
	For any nonnegative number $\sigma \geq 0$, we now choose a sequence $\{ \rho_{k} \} \subset (0, \infty)$ satisfying
	\begin{equation}
		\label{eq:rho-k-def}
		\rho_{k} \to 0^+ \quad \text{and} \quad \frac{\epsilon_k}{\rho_{k}} \to \sigma.
	\end{equation}
	Such a sequence can be chosen, for example, as follows:
	\[
	\left\{ 
	\begin{aligned}
		& \rho_k = \sqrt{\epsilon_k}, \quad k \geq 1 && \text{if } \sigma =0, \\
		& \rho_{k} = \frac{1}{\sigma} \epsilon_k, \quad k \geq 1 && \text{if } \sigma >0.
	\end{aligned}
	\right.
	\]
	To simplify the notation, from now on, for any $(f,h) \in L^2(\Omega) \times L^2(\Gamma)$ and any $k \geq 1$, we use the following abbreviations
	\begin{equation}
		\label{eq:abbreviations-k}
		\left\{
		\begin{aligned}
			& \bar u_{\pm, k} := \bar u^{\epsilon_k}_{\pm}, \quad \bar v_{\pm, k} := \bar v^{\epsilon_k}_{\pm}, \quad \bar w_{\pm, k} := (\bar u^{\epsilon_k}_{\pm}, \bar v^{\epsilon_k}_{\pm}), \quad \bar y_{\pm, k} := S(\bar u^{\epsilon_k}_{\pm}, \bar v^{\epsilon_k}_{\pm}  ), \\
			& \bar y_{\pm, k}^{f,h} := S(\bar u^{\epsilon_k}_{\pm} + \rho_k f, \bar v^{\epsilon_k}_{\pm}  + \rho_k h), \quad \bar z_{\pm, k} := S'(\bar u_{\pm, k}, \bar v_{\pm, k})(f,h), \\
			& \bar\eta_{\pm, k}^{f,h} := S(\bar u_{\pm, k} + \rho_k f, \bar v_{\pm,k}  + \rho_k h)  - S(\bar u_{\pm,k}, \bar v_{\pm,k}  )  - \rho_k \bar z_{\pm, k}.
		\end{aligned}
		\right.
	\end{equation}

We now introduce the following sets in $H^1(\Omega)$ 
	(depending on $(f,h)$ and $\sigma$):
\begin{multline}
	\label{eq:W-pm-sets}
	W_{\pm}^\sigma(\bar u, \bar v;f, h)  
	:=  \{ e  \mid e \, \text{is a weak accumulation in $H^1(\Omega)$ of $\frac{1}{\rho_k}\bar\eta_{\pm, k}^{f,h}$}
		 \}.
\end{multline}

\medskip 

Thanks to \cref{prop:control-to-state-app}~\ref{item:S-Lipschitz}, \eqref{eq:vepsilon-density-minus}, and \eqref{eq:vepsilon-density-plus}, we arrive at the following estimates, 
	which help us to prove the nonemptiness of $W_{\pm}^\sigma(\bar u, \bar v;f, h)$.

\begin{lemma} \label{lem:distance-y-epsilon-rho}
		Assume that \eqref{eq:uv-bound-assumption} is fulfilled for $(u, v) :=(\bar u, \bar v)$.
	Then, there exists a constant $C := C(L) >0$ independent of $(\bar u, \bar v)$, $k$, $\sigma$, $f$, and $h$ such that
	\[
		\norm{ \bar y_{-, k} - \bar y}_{H^1(\Omega)}  + \norm{\bar y_{-, k} - \bar y}_{C(\overline\Omega)} \leq C \epsilon_k
	\]
	and
	\[
		\norm{ \bar y_{-, k}^{f,h} - \bar y_{-, k} }_{H^1(\Omega)}  + \norm{\bar y_{-, k}^{f,h} - \bar y_{-, k}}_{C(\overline\Omega)} \leq C \rho_k\left(\norm{f}_{L^2(\Omega)} +  \norm{h}_{L^2(\Gamma)}\right).
	\]
	Additionally, if $(\bar u, \bar v) \neq (u_b, v_b)$, then the last two estimates also hold for the subscript "$+$" in place of the subscript "$-$".
\end{lemma}

\medskip


\medskip

The nonemptiness of the sets $W_{\pm}^\sigma(\bar u,\bar v;f, h)$ is presented in the following.
\begin{proposition}
	\label{prop:W-pm-emptiness}
	For any $(f,h) \in  L^2(\Omega) \times L^2(\Gamma)$, there holds $W_{-}^\sigma(\bar u,\bar v; f, h) \neq \emptyset$. Moreover, if, in addition, $(\bar u,\bar v) \neq (u_b,v_b)$, then $W_{+}^\sigma(\bar u,\bar v; f, h) \neq \emptyset$.
\end{proposition}
\begin{proof}
	First, thanks to \cref{prop:G-pm-belongto-Bouligand-diff}, $\{S'(\bar u_{-,k}, \bar v_{-,k})\}$ is bounded in $\Linop(L^2(\Omega) \times L^2(\Gamma), H^1(\Omega) \cap C(\overline\Omega))$. From this,  \cref{lem:distance-y-epsilon-rho}, and the definition of $\bar\eta_{\pm, k}^{f,h}$ in \eqref{eq:abbreviations-k}, we have
	\begin{equation*}
		\frac{1}{\rho_k} \norm{ \bar\eta_{\pm, k}^{f,h} }_{H^1(\Omega)} 
		 \leq C  \left(\norm{f}_{L^2(\Omega)} +  \norm{h}_{L^2(\Gamma)}\right)
	\end{equation*}
	for some constant $C$ independent of $(f,h)$, $k$, and $\sigma$, which, together with the reflexivity of $H^1(\Omega)$, shows the nonemptiness of $W_{-}^\sigma(\bar u,\bar v; f, h)$. Similarly, $W_{+}^\sigma(\bar u,\bar v; f, h) \neq \emptyset$ provided that $(\bar u,\bar v) \neq (u_b,v_b)$. 
\end{proof}


\medskip

	Next, we define the following nonnegative functions depending on $(\bar u, \bar v)$:
	\begin{subequations}
		\label{eq:variational-funcs}
		\begin{align}
			& \tilde{u}_{-} := \1_{\Omega}, \quad && \tilde{v}_{-} := \1_{\Gamma}, \label{eq:variational-func-minus} \\
			\intertext{and}
			& \tilde{u}_{+} := 
			\left\{
			\begin{aligned}
				& u_b - \bar u && \text{if } u_b \in L^2(\Omega),\\
				& \1_{\Omega} && \text{if } u_b = \infty,
			\end{aligned}
			\right.
			&& \tilde{v}_{+} := 
			\left\{
			\begin{aligned}
				& v_b - \bar v && \text{if } v_b \in L^2(\Gamma),\\
				& \1_{\Gamma} && \text{if } v_b = \infty.
			\end{aligned}
			\right. \label{eq:variational-func-plus} 
		\end{align}
	\end{subequations}

In the following we will derive 
	a characterization of 
the sets $W_{\pm}^\sigma(\bar u,\bar v;f,h)$, which 
	reveals that these sets are singletons identifying the difference between directional and Bouligand generalized derivatives of the control-to-state mapping. 
\begin{lemma}
	\label{lem:W-pm-sets-characterization}
	Let $\sigma \geq 0$ be arbitrary, but fixed. Then, the following assertions are fulfilled:
	\begin{enumerate}[label=(\alph*)]
		\item \label{item:W-set-minus-char} 
		For any $(f,h) \in L^2(\Omega) \times L^2(\Gamma)$, there holds,
		\begin{equation}
			\label{eq:W-minus-characterization}
			W_{-}^\sigma(\bar u,\bar v; f, h) = \{ S'(\bar u, \bar v; f - \sigma \tilde{u}_{-},h - \sigma \tilde{v}_{-}) - \bar G_{-}(f - \sigma \tilde{u}_{-},h - \sigma \tilde{v}_{-}) \}.
		\end{equation}
		\item \label{item:W-set-plus-char}  If, in addition $(\bar u,\bar v )\neq (u_b,v_b)$, then, for any $(f,h) \in L^2(\Omega) \times L^2(\Gamma)$,
		\begin{equation}
			\label{eq:W-plus-characterization}
				W_{+}^\sigma(\bar u,\bar v; f, h) = \{ S'(\bar u, \bar v; f + \sigma \tilde{u}_{+},h + \sigma \tilde{v}_{+}) - \bar G_{+}(f + \sigma \tilde{u}_{+},h + \sigma \tilde{v}_{+}) \}.
		\end{equation}
	\end{enumerate}	
	Here  $\bar G_{\pm}$ and $\tilde{u}_{\pm}$, $\tilde{v}_{\pm}$ are given in \eqref{eq:G-pm-operator} and \eqref{eq:variational-funcs}, respectively. 
\end{lemma}
\begin{proof}
	\noindent\emph{Ad \ref{item:W-set-minus-char}}.
	First, take $e_{-} \in W_{-}^\sigma(\bar u,\bar v; f, h)$ arbitrarily. By definition, there exists a subsequence, denoted in the same way, of $\{k\}$ such that \eqref{eq:epsilonk-ukvk-minus} and \eqref{eq:rho-k-def} are satisfied. Moreover, there holds
	\begin{equation}
		\label{eq:e-minus-determine}
		\frac{1}{\rho_k}\bar \eta^{f, h}_{-,k} \rightharpoonup e_{-} \, \text{ in $H^1(\Omega)$}.
	\end{equation}
	From the definition of $\bar \eta^{f, h}_{-,k}$ in  \eqref{eq:abbreviations-k}, we can rewrite $\bar \eta^{f, h}_{-,k}$ as follows
	\begin{align}
		\bar \eta^{f, h}_{-,k} & = S(\bar u_{-,k} + \rho_{k} f, \bar v_{-,k} + \rho_k h) - S(\bar u_{-,k} , \bar v_{-,k} ) - 
			\rho_k
		S'(\bar u_{-,k} , \bar v_{-,k})(f,h) \label{eq:eta-k-minus-rewriten} \\
		& = [S((\bar u, \bar v) + \rho_{k}(f_{-,k}, h_{-,k}) ) - S(\bar u, \bar v) ]+ [S(\bar u, \bar v) - S(\bar u_{-,k} , \bar v_{-,k} )] - 
			\rho_k
		\bar z_{-,k} \nonumber
	\end{align}
	with
	\[
		f_{-,k} := f + \frac{1}{\rho_k}(\bar u_{-,k} - \bar u) \quad \text{and} \quad h_{-,k} := h + \frac{1}{\rho_k}(\bar v_{-,k} - \bar v).
	\]
	By applying now \cref{prop:G-pm-belongto-Bouligand-diff} for $(u,v) := (\bar u, \bar v)$ and using the identity \eqref{eq:G-pm-bar-Guv}, there holds
	\begin{equation}
		\label{eq:zk-minus-limit}
		\bar z_{-,k} \to \bar G_{-}(f,h) \quad \text{in } H^1(\Omega).
	\end{equation}
	Besides, we deduce from \eqref{eq:variational-func-minus}, \eqref{eq:uv-varepsilon}, and \eqref{eq:abbreviations-k} that
	\begin{equation}
		\label{eq:uvk-bar-tilde-minus-expression}
		\bar u_{-, k} = \bar u - \epsilon_k \tilde{u}_{-} \quad \text{and} \quad \bar v_{-, k} = \bar v - \epsilon_k \tilde{v}_{-}.
	\end{equation}
	This, together with the second limit in \eqref{eq:rho-k-def}, then gives 
	\begin{equation}
		\label{eq:fh-k-limit-minus}
		(f_{-,k},h_{-,k}) \to (f - \sigma \tilde{u}_{-},h -\sigma \tilde{v}_{-}) \quad \text{in} \quad L^2(\Omega) \times L^2(\Gamma).
	\end{equation}
	From this and \cref{prop:control-to-state-app}~\ref{item:S-dir-der}, we have
	\begin{equation}
		\label{eq:directional-der-k}
		\frac{1}{\rho_k}[S((\bar u, \bar v) + \rho_{k}(f_{-,k}, h_{-,k}) ) - S(\bar u, \bar v) ] \to S'(\bar u, \bar v; f - \sigma \tilde{u}_{-},h -\sigma \tilde{v}_{-})
	\end{equation} 
	in $H^1(\Omega)$.
	Similarly, it follows from \eqref{eq:uvk-bar-tilde-minus-expression}, the second limit in  \eqref{eq:rho-k-def}, and \cref{prop:control-to-state-app}~\ref{item:S-dir-der} that
	\[
		\frac{1}{\rho_{k}}[S(\bar u, \bar v) - S(\bar u_{-,k} , \bar v_{-,k} )] \to  - S'(\bar u, \bar v, - \sigma \tilde{u}_{-}, -\sigma \tilde{v}_{-})
	\]
	in $H^1(\Omega)$.
	Combining this with \eqref{eq:G-pm-bar-Guv} and \eqref{eq:S-der-G-identity} yields
	\[
		\frac{1}{\rho_{k}}[S(\bar u, \bar v) - S(\bar u_{-,k} , \bar v_{-,k} )] \to  \sigma \bar G_{-}(\tilde{u}_{-}, \tilde{v}_{-})
	\]
	in $H^1(\Omega)$,
	where we have just used the pointwise nonnegativity of $\tilde{u}_{-}$ and $\tilde{v}_{-}$ defined in \eqref{eq:variational-func-minus}.
	We thus employ the expression \eqref{eq:eta-k-minus-rewriten} and the limits \eqref{eq:zk-minus-limit}, \eqref{eq:directional-der-k} to have
	\begin{equation}
		\label{eq:eta-k-minus-limit}
		\frac{1}{\rho_k}\bar \eta^{f, h}_{-,k} \to S'(\bar u, \bar v; f - \sigma \tilde{u}_{-},h -\sigma \tilde{v}_{-}) + \sigma \bar G_{-}(\tilde{u}_{-}, \tilde{v}_{-}) - \bar G_{-}(f,h) \, \text{in } H^1(\Omega),
	\end{equation}
	which, along with \eqref{eq:e-minus-determine}, implies \eqref{eq:W-minus-characterization}.

\medskip

	\noindent\emph{Ad \ref{item:W-set-plus-char}}.	In this case, the argument is analogous to that for \ref{item:W-set-minus-char} with a slight modification. Indeed, similar to \eqref{eq:eta-k-minus-rewriten}, we can rewrite $\bar \eta^{f, h}_{+,k}$ as follows
	\begin{equation}
		\bar \eta^{f, h}_{+,k}  = [S((\bar u, \bar v) + \rho_{k}(f_{+,k}, h_{+,k}) ) - S(\bar u, \bar v) ]+ [S(\bar u, \bar v) - S(\bar u_{+,k} , \bar v_{+,k} )] - 
			\rho_k
		\bar z_{+,k} \label{eq:eta-k-plus-rewriten} 
	\end{equation}
	with
	\[
		f_{+,k} := f + \frac{1}{\rho_k}(\bar u_{+,k} - \bar u) \quad \text{and} \quad h_{+,k} := h + \frac{1}{\rho_k}(\bar v_{+,k} - \bar v).
	\]
	Also, from \eqref{eq:variational-func-plus}, \eqref{eq:uv-varepsilon}, and \eqref{eq:abbreviations-k} that
	\begin{equation}
		\label{eq:uvk-bar-tilde-plus-expression}
		\bar u_{+, k} = \bar u + \epsilon_k \tilde{u}_{+} \quad \text{and} \quad \bar v_{+, k} = \bar v + \epsilon_k \tilde{v}_{+},
	\end{equation}
	which is analogous to \eqref{eq:uvk-bar-tilde-minus-expression}.
	A combination of \eqref{eq:uvk-bar-tilde-plus-expression} with the second limit in \eqref{eq:rho-k-def}  then gives 
	\begin{equation}
		\label{eq:fh-k-limit-plus}
		(f_{+,k},h_{+,k}) \to (f + \sigma \tilde{u}_{+},h + \sigma \tilde{v}_{+}) \quad \text{in} \quad L^2(\Omega) \times L^2(\Gamma).
	\end{equation}
	This is comparable to \eqref{eq:fh-k-limit-minus}. From these and the same argument showing \eqref{eq:directional-der-k}--\eqref{eq:eta-k-minus-limit}, we obtain \eqref{eq:W-plus-characterization}.
\end{proof}


The following is a direct consequence of \cref{lem:W-pm-sets-characterization}, in which the sets $W_{\pm}^\sigma(\bar u,\bar v; f, h)$ reduce to ones consisting of  one element $e_{\pm} = 0$ only, provided that some additional assumptions, in particular one ensuring the G\^{a}teaux differentiability of $S$, are imposed.
\begin{corollary}
	\label{cor:W-pm-zero}
	The following assertions are valid:
	\begin{enumerate}[label =(\alph*)]
		\item \label{item:W-set-Gateaux-diff} If $S$ is G\^{a}teaux differentiable in  $(\bar u,\bar v)$, then $W_{\pm}^\sigma(\bar u,\bar v; f, h) = \{0\}$ for all $(f,h) \in L^2(\Omega) \times L^2(\Gamma)$ and $\sigma \geq 0$.
		
		\item \label{item:W-set-minus-zero}  For any $(f,h) \in L^2(\Omega) \times L^2(\Gamma)$ with $f \leq 0$ a.a. in $\Omega$ and $h \leq 0$ a.a. on $\Gamma$, there holds 
		\begin{equation} \label{eq:W-minus-zero}
			W_{-}^\sigma(\bar u,\bar v; f,h) = \{0\} \quad \text{for all} \quad \sigma \geq 0.
		\end{equation}
			Consequently, one has
			\begin{equation} \label{eq:W-minus-zero-ubvb}
				W_{-}^\sigma(\bar u,\bar v; u- u_b, v -v_b) = \{0\} \quad \text{for all} \quad (u,v) \in U_{ad}, \,  \sigma \geq 0.
			\end{equation}
		
		\item \label{item:W-set-plus-zero} For any $(f,h) \in L^2(\Omega) \times L^2(\Gamma)$ with $f \geq 0$ a.a. in $\Omega$ and $h \geq  0$ a.a. on $\Gamma$, there holds 
		\begin{equation*} 
			W_{+}^\sigma(\bar u,\bar v;f, h) = \{0\}  \quad \text{for all} \quad \sigma \geq 0.
		\end{equation*}
	\end{enumerate}	 
\end{corollary}
\begin{proof}
	\noindent \emph{Ad \ref{item:W-set-Gateaux-diff}.} Assume that the control-to-state operator $S$ is G\^{a}teaux differentiable in  $(\bar u, \bar v)$. 
	Then, there holds $S'(\bar u, \bar v) = \bar G_{\pm}$, on account of \cref{cor:Bouligand-Gateaux}.
	From this and \cref{lem:W-pm-sets-characterization}, the sets $W_{\pm}(\bar u,\bar v;f,h)$ thus consist of one element $e_{\pm} = 0$ only. 	
	
	\medskip 
	
	\noindent \emph{Ad \ref{item:W-set-minus-zero}.} Assume now that $f \leq 0$ a.a. in $\Omega$, and $h \leq 0$ a.a. on $\Gamma$. 
		We therefore have $f - \sigma \tilde{u}_{-} \leq 0$ a.a. in $\Omega$ and $h - \sigma \tilde{v}_{-} \leq 0$ a.a. on $\Gamma$, due to the nonnegativity of $\tilde{u}_{-}$ and $\tilde{v}_{-}$ defined in \eqref{eq:variational-func-minus}.
		From this and \eqref{eq:G-pm-bar-Guv} as well as the first identity in \eqref{eq:S-der-G-identity}, there holds
		\[
		S'(\bar u, \bar v; f - \sigma \tilde{u}_{-},h - \sigma \tilde{v}_{-}) = - \bar G_{-}(-f + \sigma \tilde{u}_{-}, -h + \sigma \tilde{v}_{-}),
		\]
		which, along with the linearity of $\bar G_{-}$ and \eqref{eq:W-minus-characterization}, gives \eqref{eq:W-minus-zero}.
	
	\medskip 
	
	\noindent \emph{Ad \ref{item:W-set-plus-zero}.} The proof for this assertion is similar to \ref{item:W-set-minus-zero}.
\end{proof}

\begin{remark}
	\label{rem:ubvb}
		As it will be seen in the proof of \cref{thm:control-ub-vb} below, the identity \eqref{eq:W-minus-zero-ubvb} plays an important role in establishing the optimality condition \eqref{eq:control-ub-vb-OC}--\eqref{eq:adjoint-state-ubvb} for the special situation, where a local minimizer $(\bar u, \bar v)$ is assumed to be identical to the upper bound $(u_b, v_b)$.
\end{remark}


Next, we derive some estimates on elements of $W_{\pm}^\sigma(\bar u,\bar v; f, h)$,
	in which the upper bounds for these elements are independent of $\sigma$.
		The associated proof mainly relies on  
	the definition of $W_{\pm}^\sigma(\bar u,\bar v; f, h)$ in \eqref{eq:W-pm-sets}, except the second limit in \eqref{eq:rho-k-def}.
\begin{lemma}
	\label{lem:eta-epsilon-estimate}
		Let $e_{\pm}$ be the unique elements of $W_{\pm}^\sigma(\bar u,\bar v; f, h)$ shown as in \cref{lem:W-pm-sets-characterization}. Then, there hold:
		\begin{equation}
			\label{eq:e-minus-esti}
			\norm{e_{-}}_{H^1(\Omega)} \leq C |d_1'(\bar t) - d_2'(\bar t)|\norm{ \bar z_{-} \1_{\{ \bar z_{-} \geq 0\} } \1_{\{ \bar y = \bar t\}}}_{L^2(\Omega)}	
		\end{equation}
		and
		\begin{equation}
			\label{eq:e-plus-esti}
			\norm{e_{+}}_{H^1(\Omega)} \leq C |d_1'(\bar t) - d_2'(\bar t)| \norm{\bar z_{+} \1_{\{ \bar z_{+} \leq 0\}}  \1_{\{ \bar y = \bar t\}}}_{L^2(\Omega)} 
		\end{equation} 
		for some positive constant $C$ independent of $(f,h)$ and $\sigma$. 		
		Here $\bar z_{\pm} := \bar G_{\pm}(f,h)$. 
\end{lemma}
\begin{proof} 
	By definition, there exists a subsequence, denoted in the same way, of $\{k\}$ such that \eqref{eq:epsilonk-ukvk-minus} and \eqref{eq:rho-k-def} are satisfied.
	Moreover, one has
	\begin{equation}
		\label{eq:e-determine}
		\frac{1}{\rho_k}\bar \eta^{f, h}_{\pm,k} \rightharpoonup e_{\pm} \quad \text{ in $H^1(\Omega)$}.
	\end{equation}
	From \cref{prop:G-pm-belongto-Bouligand-diff}, we have
	\begin{equation}
		\label{eq:zk-z-limit}
		\bar z_{\pm, k} \to \bar z_{\pm} \quad \text{strongly in } H^1(\Omega) \cap C(\overline\Omega). 
	\end{equation}
	Moreover, the limits in \eqref{eq:zk-z-limit} imply that
	\begin{equation}
		\label{eq:zk-z-pointwise-esti}
		\bar z_{\pm}(x) - \tau_{k} \leq \bar z_{\pm, k}(x) \leq 
		\bar z_{\pm}(x) + \tau_{k} 
		\quad \text{for all } x \in \overline{\Omega}
	\end{equation}
	with 
	\[
		\tau_{k} := \max \{\norm{\bar z_{-, k} - \bar z_{-}}_{C(\overline{\Omega})}, \norm{\bar z_{+, k} - \bar z_{+}}_{C(\overline{\Omega})} \}  \to 0.
	\]
	Thanks to \eqref{eq:vepsilon-G-diff-minus} and \eqref{eq:vepsilon-G-diff-plus} for $u^\epsilon_{\pm} := \bar u_{\pm, k}$  and $v^\epsilon_{\pm} := \bar v_{\pm, k}$ (see, also \eqref{eq:epsilonk-ukvk-minus} and \eqref{eq:abbreviations-k}), there hold
	\begin{equation}
		\label{eq:yk-measurezero}
		\meas_{\R^2} (\{\bar y_{\pm, k} = \bar t \}) = 0 \quad \text{for all } k \geq 1.
	\end{equation}	
	Furthermore, from the definition of $\bar y_{\pm, k}^{f,h}$ in \eqref{eq:abbreviations-k}, we can write
	\begin{equation}
		\label{eq:ykh-yk-expression}
		\bar y_{\pm, k}^{f,h} - \bar y_{\pm, k} = \bar \eta^{f, h}_{\pm, k} +  \rho_k \bar z_{\pm, k}.
	\end{equation}
	By using \eqref{eq:yk-measurezero} and \eqref{eq:ykh-yk-expression}, a simple computation shows that $\bar \eta^{f, h}_{\pm, k}$ satisfies
	\begin{equation} \label{eq:eta-k-pde}
		\left \{
		\begin{aligned}
			&-\Delta \bar \eta^{f, h}_{\pm, k} = 
			- [d(\bar y_{\pm,k}^{f,h})-d(\bar y_{\pm,k} + \rho_k \bar z_{\pm,k} )] - {\omega}_{\pm,k}   
			&& \text{in} \, \Omega, \\
			& \frac{\partial \bar \eta^{f, h}_{\pm, k}}{\partial \nuv}  + b(x) \bar \eta^{f, h}_{\pm, k} =  0&& \text{on}\, \Gamma
		\end{aligned}
		\right.
	\end{equation}
	with 
	\begin{equation}
		\label{eq:omegak-def}
		\omega_{\pm, k} := d(\bar y_{\pm, k} + \rho_k \bar z_{\pm,k})  - d(\bar y_{\pm, k}) - \rho_k d'(\bar y_{\pm, k}) \bar z_{\pm,k}.
	\end{equation}	
	From the expression of $\bar y_{\pm,k}^{f,h}$ in \eqref{eq:ykh-yk-expression} and the increasing monotonicity of $d$; see \cref{ass:d-func-nonsmooth}, we have
	\[
		[d(\bar y_{\pm,k}^{f,h})-d(\bar y_{\pm,k} + \rho_k \bar z_{\pm,k} )]\bar \eta^{f, h}_{\pm,k} \geq 0 \quad \text{a.a. in } \Omega.
	\]	
	Testing \eqref{eq:eta-k-pde} by $\bar \eta^{f, h}_{\pm,k}$ and using the Cauchy--Schwarz inequality, we  then obtain
	\begin{equation}
		\label{eq:eta-k-esti-H1}
		\norm{\bar \eta^{f, h}_{\pm,k}}_{H^1(\Omega)} \leq C \norm{{\omega}_{\pm,k}}_{L^2(\Omega)}
	\end{equation}
	for some constant $C>0$ independent of $k$, $(f,h)$, and $\sigma$.	
	For all $k \geq 1$,  define the sets
	\begin{equation}
		\label{eq:Omega-pm-sets}
		\left\{
			\begin{aligned}
				& \Omega_{\pm,k}^1 := \{ 
				\bar y_{\pm, k} + \rho_k \bar z_{\pm,k}
				\leq \bar t \} \cap \{ \bar y_{\pm,k} < \bar t \}, \\
				&  \Omega_{\pm,k}^2 := \{ 
					\bar y_{\pm, k} + \rho_k \bar z_{\pm,k}
				\leq \bar t \} \cap \{ \bar y_{\pm,k} > \bar t \},\\
				& \Omega_{\pm,k}^3 := \{ 
					\bar y_{\pm, k} + \rho_k \bar z_{\pm,k}
				> \bar t \} \cap \{ \bar y_{\pm,k} < \bar t \}, \\
				& \Omega_{\pm,k}^4 := \{
					\bar y_{\pm, k} + \rho_k \bar z_{\pm,k}
				> \bar t \} \cap \{ \bar y_{\pm,k} > \bar t \}.
			\end{aligned}
		\right.
	\end{equation}	
	It is easy to see that the sets $\Omega_{\pm,k}^i$, $1 \leq i \leq 4$, are pairwise disjoint and there hold
	\begin{equation}
		\label{eq:sum-Omegak-i}
		\sum_{i=1}^4 \1_{\Omega_{\pm,k}^i} = 1 \quad \text{for all} \quad k \geq 1
	\end{equation}
	a.a. in $\Omega$,
	as a result of \eqref{eq:yk-measurezero}.

	\medskip

	We now prove \eqref{eq:e-minus-esti} and \eqref{eq:e-plus-esti}.

	\noindent $\star$ \emph{Step 1: Showing \eqref{eq:e-minus-esti}}.
	To this end, we deduce from \cref{lem:strong-maximum-prin} and \eqref{eq:uepsilon-admissible-minus}--\eqref{eq:vepsilon-admissible-minus} that
	\[
		\bar y_{-,k}(x) < \bar y(x) \quad \text{for all } x \in \Omega,
	\]
	which, together with \cref{lem:distance-y-epsilon-rho} and the continuity of $\bar y$ on $\overline\Omega$, yields
	\begin{equation}
		\label{eq:yk-less-y}
		\bar y(x) - C\epsilon_k \leq \bar y_{-,k}(x) \leq \bar y(x) \quad \text{for all } x \in \overline \Omega.
	\end{equation} 
	We now estimate $\omega_{-,k}$, defined in \eqref{eq:omegak-def}, on each $\Omega_{-,k}^i$ (given in \eqref{eq:Omega-pm-sets}) with $1 \leq i \leq 4$.
	
	\noindent $\bullet$ In $\Omega_{-,k}^1$, we have
	\begin{align*}
		\omega_{-,k} &= d_1(\bar y_{-,k} + \rho_k \bar z_{-,k})  - d_1(\bar y_{-,k}) - \rho_{k} d_1'(\bar y_{-,k}) \bar z_{-,k}\\
		& = \rho_k \int_0^1[ d_1'(\bar y_{-,k} +  \theta \rho_k \bar z_{-,k})  - d_1'(\bar y_{-,k}) ] \bar z_{-,k}d\theta.
	\end{align*}	
	By setting 
	\[
		\gamma_k(x,\theta) := d_1'(\bar y_{-,k}(x) + \theta \rho_k \bar z_{-,k}(x))  - d_1'(\bar y_{-,k}(x))
	\]	
	for a.a. $x \in \Omega$ and for all $0 \leq \theta \leq 1$, there thus holds
	\begin{align}
		\norm{\1_{\Omega_{-,k}^1} \omega_{-,k} }_{L^2(\Omega)} & \leq \rho_k \norm{\bar z_{-,k}}_{C(\overline\Omega)} \norm{\int_0^1 \gamma_k(\cdot, \theta) d\theta }_{L^2(\Omega)} \nonumber \\
		& \leq \rho_k \norm{\bar z_{-,k}}_{C(\overline\Omega)}  \norm{\gamma_k}_{L^2(\Omega \times (0,1))}. \label{eq:omegak-1}
	\end{align}
	Thanks to \cref{lem:distance-y-epsilon-rho} and \eqref{eq:zk-z-limit}, there exists a constant $C_0$ independent of $k$ such that
	\[
		\norm{\bar z_{-,k}}_{C(\overline\Omega)}, \norm{\bar y_{-,k}}_{C(\overline\Omega)} \leq C_0 \quad \text{for all } k \geq 1.
	\]
	Setting $C_1 := \max\{ |d_1'(\tau)| : -2C_0 \leq \tau \leq 2C_0 \}$ yields
	\[
		|\gamma_k(x, \theta)| \leq 2C_1 \quad \text{for a.a. } (x, \theta) \in \Omega \times (0,1) \quad \text{and for all } k \geq 1.
	\]
	On the other hand, in view of \cref{lem:distance-y-epsilon-rho}, one has
	$\bar y_{-,k} \to \bar y$ in $C(\overline{\Omega})$ as $k \to \infty$. From this, \eqref{eq:zk-z-limit}, and the uniform continuity of $d_1'$ on $[-2C_0,2C_0]$, there holds
	\[
		\gamma_k(x,\theta) \to 0 \quad \text{as} \quad k \to \infty 
	\]
	for a.a. $(x,\theta) \in \Omega \times (0,1)$. We then deduce from Lebesgue's dominated convergence theorem that
	\begin{equation*}
		\norm{\gamma_k}_{L^2(\Omega \times (0,1))} \to 0 \quad \text{as} \quad k \to \infty, 
	\end{equation*}
	which, along with \eqref{eq:zk-z-limit} and \eqref{eq:omegak-1},
	yields
	\begin{equation}
		\label{eq:omega-1-limit}
		\frac{1}{\rho_k} \norm{\1_{\Omega_{-,k}^1} \omega_{-,k} }_{L^2(\Omega)} \to 0 \quad \text{as} \quad k \to \infty.
	\end{equation}
	
	\noindent $\bullet$ On $\Omega_{-,k}^4$, one similarly has
	\begin{equation}
		\label{eq:omega-4-limit}
		\frac{1}{\rho_k} \norm{\1_{\Omega_{-,k}^4} \omega_{-,k} }_{L^2(\Omega)} \to 0 \quad \text{as} \quad k \to \infty.
	\end{equation}
	
	\noindent $\bullet$ In $\Omega_{-,k}^2$, there holds 
	\begin{multline*}
		\omega_{-,k} = d_1(\bar y_{-,k} + \rho_k \bar z_{-,k} )  -  d_2(\bar y_{-,k}) - \rho_k d_2'(\bar y_{-,k})  \bar z_{-,k} \\
		\begin{aligned}
			& =\underbrace{[d_1(\bar y_{-,k} + \rho_k \bar z_{-,k} ) -d_2(\bar y_{-,k}  + \rho_k \bar z_{-,k} )]}_{ =:B_{-,k}} \\
			& \qquad + \underbrace{ [d_2(\bar y_{-,k} + \rho_k \bar z_{-,k} )  -  d_2(\bar y_{-,k}) - \rho_k d_2'(\bar y_{-,k}) \bar z_{-,k}]}_{=: A_{-,k} }.
		\end{aligned} 
	\end{multline*}
	Analogous to \eqref{eq:omega-1-limit}, we derive
	\[
		\frac{1}{\rho_k} \norm{ A_{-,k} }_{L^2(\Omega)} \to 0 \quad \text{as} \quad k \to \infty.
	\]
	This implies that
	\begin{equation}
		\label{eq:omega-2-limit-esti1}
		\limsup\limits_{k \to \infty}\frac{1}{\rho_k} \norm{\1_{\Omega_{-,k}^2} \omega_{-,k} }_{L^2(\Omega)} \leq \limsup\limits_{k \to \infty}\frac{1}{\rho_k} \norm{\1_{\Omega_{-,k}^2} B_{-,k} }_{L^2(\Omega)}.
	\end{equation}	
	On the other hand, by using \eqref{eq:d12-continuity}, the triangle inequality, and the bounds of $\{\bar y_{-,k} \}$ and $\{\bar z_{-,k}\}$ in $C(\overline\Omega)$ resulting from \cref{lem:distance-y-epsilon-rho} and \eqref{eq:zk-z-limit}, as well as the continuous differentiability of $d_i$, $i=1,2$, we have
	\begin{multline}
		\label{eq:Bk-omega2}
		|B_{-,k}|  = |d_1(\bar y_{-,k} + \rho_k \bar z_{-,k} ) - d_1(\bar t)  + d_2(\bar t) - d_2(\bar y_{-,k} + \rho_k \bar z_{-,k}  )| \\
		\begin{aligned}[b]
			& \leq |d_1(\bar y_{-,k} + \rho_k \bar z_{-,k}) - d_1(\bar t)| + | d_2(\bar t) - d_2(\bar y_{-,k} + \rho_k \bar z_{-,k}  )| \\
			& \leq C_{f,h}|\bar y_{-,k} + \rho_k \bar z_{-,k} - \bar t|  = C_{f,h}(\bar t- \bar y_{-,k} - \rho_k \bar z_{-,k}) \\
			& \leq -\rho_k C_{f,h}\bar z_{-,k}
		\end{aligned}
	\end{multline}	
	a.a. in $\Omega^2_{-,k}$ and for some positive constant $C_{f,h}$ independent of $k$,
	where we have used the definition of $\Omega^2_{-,k}$ to obtain the last  identity and the last inequality.
	Again, exploiting the definition of $\Omega^2_{-,k}$,  we then conclude from 
	\eqref{eq:zk-z-pointwise-esti} and \eqref{eq:yk-less-y} that
	\begin{align*}
		\Omega_{-,k}^2 & = \{\bar y_{-,k} > \bar t \} \cap \{ \bar y_{-,k}  + \rho_k \bar z_{-,k} \leq \bar t \}\\
		& \subset \{ \bar y > \bar t \} \cap \{ \bar y - C \epsilon_k \leq \bar t - \rho_k (\bar z_{-} - \tau_k) \} \\
		& = \{ \bar t < \bar y \leq \bar t + C \epsilon_k - \rho_k (\bar z_{-} - \tau_k) \}
	\end{align*}
	for all $k \geq 1$. 
	From this and the limits $\tau_k, \epsilon_k, \rho_k \to 0^+$, one has
	\[
		\1_{\Omega_{-,k}^2} \to 0 \quad \text{a.a. in $\Omega$ as $k \to \infty$. }
	\]
	This, \eqref{eq:Bk-omega2}, \eqref{eq:zk-z-limit}, and Lebesgue's dominated convergence theorem thus imply
	\[
		\frac{1}{\rho_k}\1_{\Omega_{-,k}^2}|B_{-,k}| \to 0 \quad \text{in} \quad L^2(\Omega).
	\]
	Combining this with \eqref{eq:omega-2-limit-esti1} yields
	\begin{equation}
		\label{eq:omega-2-limit}
		\frac{1}{\rho_k} \norm{\1_{\Omega_{-,k}^2} \omega_{-,k} }_{L^2(\Omega)} \to 0 \quad \text{as} \quad k \to \infty.
	\end{equation}

	\noindent $\bullet$ In $\Omega_{-,k}^3$, we now have
	\begin{multline} \label{eq:omega3-Bk-Dk}
		\omega_{-,k} = d_2(\bar y_{-,k} + \rho_k \bar z_{-,k})  - d_1(\bar y_{-,k}) - \rho_k d_1'(\bar y_{-,k}) \bar z_{-,k} \\
		\begin{aligned}
			& = \underbrace{[d_2(\bar y_{-,k} + \rho_k \bar z_{-,k}) - d_1(\bar y_{-,k} + \rho_k \bar z_{-,k})]}_{= -B_{-,k}} \\
			& \qquad + \underbrace{ [d_1(\bar y_{-,k} + \rho_k \bar z_{-,k}) - d_1(\bar y_{-,k}) - \rho_k d_1'(\bar y_{-,k}) \bar z_{-,k}]}_{=: D_{-,k} }.
		\end{aligned} 
	\end{multline}
	Similar to $A_{-,k}$, one has 
	\begin{equation}
		\label{eq:Dk-limit}
		\frac{1}{\rho_k} \norm{ D_{-,k} }_{L^2(\Omega)} \to 0 \quad \text{as} \quad k \to \infty.
	\end{equation}
	We now analyze $B_{-,k}$ differently from \eqref{eq:Bk-omega2}. In fact, by employing \eqref{eq:d12-continuity}, we have
	\begin{multline} \label{eq:Bk-omega3-esti}
		B_{-,k} = [d_1'(\bar t) - d_2'(\bar t)]\underbrace{(\bar y_{-,k} + \rho_k \bar z_{-,k}- \bar t)}_{=:E_{-,k}}  \\
		\begin{aligned}[b]
			& - [d_2(\bar y_{-,k} + \rho_k \bar z_{-,k} ) - d_2(\bar t) - d_2'(\bar t)(\bar y_{-,k}+ \rho_k \bar z_{-,k}- \bar t)] \\
			& +[d_1(\bar y_{-,k} + \rho_k \bar z_{-,k} ) - d_1(\bar t) - d_1'(\bar t)(\bar y_{-,k} + \rho_k \bar z_{-,k}- \bar t)].
		\end{aligned}
	\end{multline}
	Besides, a Taylor's expansion gives
	\begin{multline*}
		|d_i(\bar y_{-,k} + \rho_k \bar z_{-,k} ) - d_i(\bar t) - d_i'(\bar t)(\bar y_{-,k} + \rho_k \bar z_{-,k}- \bar t)| \\
		\begin{aligned}
			& = \left| \int_0^1 (d_i'(\bar t + \theta (\bar y_{-,k} + \rho_k \bar z_{-,k}- \bar t) ) - d_i'(\bar t))( \bar y_{-,k} + \rho_k \bar z_{-,k}- \bar t) d\theta \right| \\
			& \leq |\bar y_{-,k} + \rho_k \bar z_{-,k}- \bar t|  \int_0^1 |d_i'(\bar t + \theta (\bar y_{-,k} + \rho_k \bar z_{-,k}- \bar t) ) - d_i'(\bar t)| d\theta
		\end{aligned}
	\end{multline*}
	a.a. in $\Omega$
	with $i =1,2$. From the definition of $\Omega_{-,k}^3$, there holds
	\begin{equation}
		\label{eq:omega3-esti}
		0 < \bar y_{-,k} + \rho_k \bar z_{-,k}- \bar t < \rho_k \bar z_{-,k} \quad \text{a.a. in} \quad \Omega_{-,k}^3.
	\end{equation}
	This, along with the limit \eqref{eq:zk-z-limit},  shows
	\[
		\1_{\Omega_{-,k}^3}|\bar y_{-,k} + \rho_{k} \bar z_{-,k}- \bar t| \to 0 \quad \text{a.a. in} \quad \Omega.
	\]
	There then holds
	\begin{multline*}
		\frac{1}{\rho_k}\1_{\Omega_{-,k}^3}|d_i(\bar y_{-,k} + \rho_k \bar z_{-,k} ) - d_i(\bar t) - d_i'(\bar t)(\bar y_{-,k} + \rho_k \bar z_{-,k}- \bar t)| \\
		\leq \1_{\Omega_{-,k}^3} \bar z_{-,k} \int_0^1 |d_i'(\bar t + \theta (\bar y_{-,k} + \rho_k \bar z_{-,k}- \bar t) ) - d_i'(\bar t)| d\theta \to 0
	\end{multline*}
	a.a. in $\Omega$ with $i=1,2$. From this and Lebesgue's dominated convergence theorem, we deduce from \eqref{eq:Bk-omega3-esti} that
	\begin{equation}
		\label{eq:Bk-Ek-limit}
			\lim\limits_{k \to \infty}
		\frac{1}{\rho_k} \norm{\1_{\Omega_{-,k}^3} \{B_{-,k} -[d_1'(\bar t) - d_2'(\bar t)] E_{-,k}\} }_{L^2(\Omega)} = 0.
	\end{equation}
	Combining this with 
	\eqref{eq:omega3-Bk-Dk} and  \eqref{eq:Dk-limit} yields
	\begin{equation*} 
			\lim\limits_{k \to \infty}
		\frac{1}{\rho_k} \norm{\1_{\Omega_{-,k}^3} \{\omega_{-,k} +[d_1'(\bar t) - d_2'(\bar t)] E_{-,k}\} }_{L^2(\Omega)} = 0,
	\end{equation*}	
	which, along with limits \eqref{eq:omega-1-limit},  \eqref{eq:omega-4-limit}, and \eqref{eq:omega-2-limit}, as well as the identity in \eqref{eq:sum-Omegak-i}, gives
	\begin{equation}
		\label{eq:omegak-limit-all}
			\lim\limits_{k \to \infty}
		\frac{1}{\rho_k} \norm{\omega_{-,k} +\1_{\Omega_{-,k}^3}[d_1'(\bar t) - d_2'(\bar t)] E_{-,k} }_{L^2(\Omega)} = 0.
	\end{equation}	
	From this, \eqref{eq:e-determine}, and \eqref{eq:eta-k-esti-H1}, we deduce from the weak lower semicontinuity of the norm in $H^1(\Omega)$ that
	\begin{equation}
		\label{eq:e-minus-H1-esti}
		\norm{e_{-}}_{H^1(\Omega)} \leq C |d_1'(\bar t) - d_2'(\bar t)|\limsup\limits_{k \to \infty}\frac{1}{\rho_k}\norm{ \1_{ \Omega_{-,k}^3}  E_{-,k}}_{L^2(\Omega)}.
	\end{equation}	
	Besides, we obtain from \eqref{eq:omega3-esti} and the definition of $E_{-,k}$ in \eqref{eq:Bk-omega3-esti} that
	\begin{equation} \label{eq:Ek-esti-pointwise}
		0 \leq \frac{1}{\rho_k} \1_{ \Omega_{-,k}^3} E_{-,k} \leq  \1_{ \Omega_{-,k}^3}\bar z_{-,k} \quad \text{a.a. in } \Omega.
	\end{equation}
%
	On the other hand, we can see from \eqref{eq:zk-z-pointwise-esti}  and \eqref{eq:yk-less-y} as well as from the definition of $\Omega_{-,k}^3$ that
	\begin{align*}
			\Omega_{-,k}^3 & = \{ \bar t - \rho_k \bar z_{-,k} < \bar y_{-,k} < \bar t \} \cap \{  \bar z_{-,k} > 0 \} \\
				& \subset \{ \bar t - \rho_k (\bar z_{-} + \tau_k) < \bar y < \bar t + C \epsilon_k \} \cap \{   \bar z_{-} > - \tau_k \}.
	\end{align*}		
	This, along with \eqref{eq:Ek-esti-pointwise}, implies that
	\begin{align*} 	
		0 \leq & \frac{1}{\rho_k} \1_{\Omega_{-,k}^3} E_{-,k} 
			\leq  \1_{ \{ \bar t - \rho_k (\bar z_{-} + \tau_k) < \bar y < \bar t + C \epsilon_k \}} \1_{  \{  \bar z_{-} > - \tau_k \}} \bar z_{-,k} =: m_k 
	\end{align*}
	a.a. in $\Omega$. 	
		Consequently, there holds
		\[
			\limsup\limits_{k \to \infty}\frac{1}{\rho_k}  \norm{\1_{\Omega_{-,k}^3} E_{-,k} }_{L^2(\Omega)} \leq \lim\limits_{k \to \infty} \norm{m_k}_{L^2(\Omega)} = \norm{\1_{\{ \bar y = \bar t \}} \1_{ \{\bar z_{-} \geq 0\}} \bar z_{-}}_{L^2(\Omega)},
		\]
		where we have used the limits $\rho_k, \epsilon_k, \tau_k \to 0$ and \eqref{eq:zk-z-limit} as well as Lebesgue's dominated convergence theorem to derive the last limit. 
		From this and \eqref{eq:e-minus-H1-esti}, we have \eqref{eq:e-minus-esti}.
%
%
%
%
%

	\medskip

	\noindent $\star$ \emph{Step 2: Showing \eqref{eq:e-plus-esti}}. 	
	The proof for \eqref{eq:e-plus-esti} is similar to that for \eqref{eq:e-minus-esti} with some slight modifications as shown below.  We first have from \eqref{eq:uepsilon-admissible-plus}--\eqref{eq:vepsilon-admissible-plus} and from \cref{lem:strong-maximum-prin} that
	\begin{equation}
		\label{eq:yk-greater-y-plus}
		\bar y(x) + C\epsilon_k \geq \bar y_{+,k}(x) \geq \bar y(x) \quad \text{for all } x \in \overline \Omega,
	\end{equation} 
	similar to \eqref{eq:yk-less-y}. We now estimate the term $\omega_{+,k}$ defined in \eqref{eq:omegak-def} analogously to $\omega_{-,k}$.

	\noindent $\bullet$ On $\Omega_{+,k}^1$, analogous to \eqref{eq:omega-1-limit}, there holds
	\begin{equation}
		\label{eq:omega-1-limit-plus}
		\frac{1}{\rho_k} \norm{\1_{\Omega_{+,k}^1} \omega_{+,k} }_{L^2(\Omega)} \to 0 \quad \text{as} \quad k \to \infty.
	\end{equation}

	\noindent $\bullet$ On $\Omega_{+,k}^4$, one  has
	\begin{equation}
		\label{eq:omega-4-limit-plus}
		\frac{1}{\rho_k} \norm{\1_{\Omega_{+,k}^4} \omega_{+,k} }_{L^2(\Omega)} \to 0 \quad \text{as} \quad k \to \infty,
	\end{equation}
	corresponding to \eqref{eq:omega-4-limit}.

	\noindent $\bullet$ On $\Omega_{+,k}^3$, by using \eqref{eq:yk-greater-y-plus} and the same argument as for $\Omega_{-,k}^2$ we arrive at
	\begin{equation}
		\label{eq:omega-3-limit-plus}
		\frac{1}{\rho_k} \norm{\1_{\Omega_{+,k}^3} \omega_{+,k} }_{L^2(\Omega)} \to 0 \quad \text{as} \quad k \to \infty.
	\end{equation}
	This is comparable with \eqref{eq:omega-2-limit}.

	\noindent $\bullet$ On $\Omega_{+,k}^2$, we will exploit the technique similar to that for $\Omega_{-,k}^3$. Analogous to \eqref{eq:omega3-Bk-Dk}, there holds 
	\begin{align} 
		\omega_{+,k} & = d_1(\bar y_{+,k} + \rho_k \bar z_{+,k})  - d_2(\bar y_{+,k}) - \rho_k d_2'(\bar y_{+,k}) \bar z_{+,k} \notag \\
			& = \underbrace{[d_1(\bar y_{+,k} + \rho_k \bar z_{+,k}) - d_2(\bar y_{+,k} + \rho_k \bar z_{+,k})]}_{=: B_{+,k}} \notag 
			\\ 
			& \qquad + \underbrace{ [d_2(\bar y_{+,k} + \rho_k \bar z_{+,k}) - d_2(\bar y_{+,k}) - \rho_k d_2'(\bar y_{+,k}) \bar z_+,k]}_{=: D_{+,k} } \label{eq:omega2-Bk-Dk-plus}
	\end{align}	
	a.a. on $\Omega_{+,k}^2$,
	as a result of \eqref{eq:omegak-def} and \eqref{eq:Omega-pm-sets}. For $D_{+,k}$, one has 
	\begin{equation}
		\label{eq:Dk-limit-plus}
		\frac{1}{\rho_k} \norm{ D_{+,k} }_{L^2(\Omega)} \to 0 \quad \text{as} \quad k \to \infty,
	\end{equation}
	analogous to \eqref{eq:Dk-limit}. Moreover, in the manner of \eqref{eq:Bk-omega3-esti}, one has
	\begin{multline*} 
		B_{+,k} = [d_1'(\bar t) - d_2'(\bar t)]\underbrace{(\bar y_{+,k} + \rho_k \bar z_{+,k}- \bar t)}_{=:E_{+,k}}  \\
		\begin{aligned}[b]
			& - [d_2(\bar y_{+,k} + \rho_k \bar z_{+,k} ) - d_2(\bar t) - d_2'(\bar t)(\bar y_{+,k}+ \rho_k \bar z_{+,k}- \bar t)] \\
			& +[d_1(\bar y_{+,k} + \rho_k \bar z_{+,k} ) - d_1(\bar t) - d_1'(\bar t)(\bar y_{+,k} + \rho_k \bar z_{+,k}- \bar t)],
		\end{aligned}
	\end{multline*}	
	which yields the following limit related to \eqref{eq:Bk-Ek-limit},
	\begin{equation*}
			\lim\limits_{k \to \infty}
		\frac{1}{\rho_k} \norm{\1_{\Omega_{+,k}^2} \{B_{+,k} -[d_1'(\bar t) - d_2'(\bar t)] E_{+,k}\} }_{L^2(\Omega)} = 0.
	\end{equation*}	
	From this, \eqref{eq:omega-1-limit-plus}--\eqref{eq:Dk-limit-plus}, and \eqref{eq:sum-Omegak-i}, we arrive at 
	\begin{equation*}
			\lim\limits_{k \to \infty}
		\frac{1}{\rho_k} \norm{\omega_{+,k} - \1_{\Omega_{+,k}^2}[d_1'(\bar t) - d_2'(\bar t)] E_{+,k} }_{L^2(\Omega)} = 0.
	\end{equation*}	
	Similarly to \eqref{eq:e-minus-H1-esti}, one therefore has 
	\begin{equation*}
		\norm{e_{+}}_{H^1(\Omega)} \leq C |d_1'(\bar t) - d_2'(\bar t)|\limsup\limits_{k \to \infty}\frac{1}{\rho_k}\norm{ \1_{ \Omega_{+,k}^2}  E_{+,k}}_{L^2(\Omega)}.
	\end{equation*}	
	This, together with the same argument using in the end of Step 1, gives \eqref{eq:e-plus-esti}.
\end{proof}

\medskip 
The following proposition states the  optimality conditions in terms of the left and right Bouligand generalized derivatives $\bar G_{\pm}$, which belong to  $\partial_B S(\bar u, \bar v)$, for the problem \eqref{eq:P}. 
	In order to prove these optimality conditions, we shall exploit the definition of the sets $W_{\pm}^\sigma$ in \eqref{eq:W-pm-sets}.
\begin{proposition}
	\label{prop:primal-OCs}
	Let $(\bar u, \bar v)$ be a local minimizer of \eqref{eq:P} and $\bar G_{\pm}$  defined  in \eqref{eq:G-pm-operator}. 
		Then, there exists a constant $C_* >0$ independent of $(\bar u, \bar v)$ such that,	for any $\sigma \geq 0$ and
	$(u,v) \in U_{ad}$, there hold:
	\begin{enumerate}[label=(\alph*)]
		\item \label{item:primal-OCs-minus}
		For $e_{-}$ 
			being the unique element in
		$W_{-}^\sigma(\bar u, \bar v; u - \bar u, v - \bar v)$,
		\begin{multline}
			\label{eq:primal-OCs-minus}
			\scalarprod{\bar y - y_\Omega}{\bar G_{-}(u- \bar u, v - \bar v)}_{\Omega} +\alpha \scalarprod{\bar y - y_\Gamma}{\bar G_{-}(u- \bar u, v - \bar v)}_{\Gamma} + \kappa_\Omega \scalarprod{\bar u}{u- \bar u}_{\Omega} \\
			 + \kappa_\Gamma \scalarprod{\bar v}{v - \bar v}_{\Gamma} \geq 
			 - C_* \sigma
			  - \left[ \scalarprod{\bar y - y_\Omega}{e_{-}}_{\Omega} + \alpha \scalarprod{\bar y - y_\Gamma}{e_{-}}_{\Gamma} \right];
		\end{multline}
		
		\item \label{item:primal-OCs-plus} If, in addition, $(\bar u,\bar v) \neq (u_b, v_b)$, then one further has
		\begin{multline*}
			\scalarprod{\bar y - y_\Omega}{\bar G_{+}(u- \bar u, v - \bar v)}_{\Omega} +\alpha \scalarprod{\bar y - y_\Gamma}{\bar G_{+}(u- \bar u, v - \bar v)}_{\Gamma} + \kappa_\Omega \scalarprod{\bar u}{u- \bar u}_{\Omega} \\
			+ \kappa_\Gamma \scalarprod{\bar v}{v - \bar v}_{\Gamma} \geq 
				- C_* \sigma
		 - \left[ \scalarprod{\bar y - y_\Omega}{e_{+}}_{\Omega} + \alpha \scalarprod{\bar y - y_\Gamma}{e_{+}}_{\Gamma} \right]
		\end{multline*}
		for  $e_{+}$  
					being the unique element in $W_{+}^\sigma(\bar u,\bar v; u - \bar u, v - \bar v)$.
	\end{enumerate}
\end{proposition}

\begin{proof}
	Since the argument for \ref{item:primal-OCs-plus} is analogous to that of \ref{item:primal-OCs-minus}, we now provide the detailed argument showing \ref{item:primal-OCs-minus} only. To this end, we first observe that
	\begin{equation}
		\label{eq:uv-bar-bound}
		\frac{\kappa_\Omega}{2}\norm{\bar u}_{L^2(\Omega)}^2 + \frac{ \kappa_\Gamma}{2} \norm{\bar v}_{L^2(\Gamma)}^ 2 \leq J(u_b, v_b).
	\end{equation}
	This, in association with \eqref{eq:apriori-esti-state}, yields
	\begin{equation}
		\label{eq:y-bar-bound}
		\norm{\bar y}_{H^1(\Omega)} + \norm{\bar y}_{C(\overline\Omega)} \leq C_0
	\end{equation}
	for some constant $C_0$.
	For any $(u,v) \in U_{ad}$, we set $\psi := (u -\bar u, v - \bar v)$. 
		Let $e_{-}$ be the unique element of $W_{-}^\sigma(\bar u, \bar v; u - \bar u, v - \bar v)$ due to \cref{lem:W-pm-sets-characterization}.
	By \eqref{eq:epsilonk-ukvk-minus}--\eqref{eq:W-pm-sets}, 
	there exists a subsequence, denoted in the same way, of $\{k\}$ such that $\epsilon_k, \rho_{k} \to 0^+$, $\bar w_{-, k} = (\bar u_{-,k}, \bar v_{-,k}) \in D_S \cap U_{ad}$ and
	\begin{subequations}
		\label{eq:limits-wk-primal}
		\begin{align}
			& 
			\frac{\epsilon_k}{\rho_{k}} \to \sigma, 
			\label{eq:epsilon-over-rho-k-primal} \\
			& \bar u_{-,k} \leq \bar u + \epsilon_k (u_b- \bar u), \quad  \bar v_{-,k} \leq \bar v + \epsilon_k (v_b- \bar v), \label{eq:uk-Uad} \\
			& \norm{\bar u_{-,k} - \bar u}_{L^2(\Omega)} + \norm{\bar v_{-,k} - \bar v}_{L^2(\Gamma)} \leq C_1\epsilon_k, \label{eq:uk-uv-primal} \\
			& \frac{1}{\rho_{k}} \bar \eta^{\psi}_{-,k} \quad \text{converges to} \quad e_{-} \quad \text{weakly in $H^1(\Omega)$}. \label{eq:eta-k-limit-primal}
		\end{align}
	\end{subequations}
	Here the constant $C_1$ does not depend on $(\bar u, \bar v)$, $(u, v)$,  $\sigma$, and $k$, due to \eqref{eq:uv-bar-bound} and \cref{lem:countable-sets}.
	By definition of $\bar \eta^{\psi}_{-,k}$ (see \eqref{eq:abbreviations-k}), we have
	\[
		\bar \eta^{\psi}_{-,k} = S(\bar w_{-,k} + \rho_k \psi) - S(\bar w_{-,k}) - \rho_k S'(\bar w_{-,k}) \psi =\bar y^{\psi}_{-,k} - \bar y_{-,k} - \rho_k S'(\bar w_{-,k}) \psi.
	\]	
	It follows from \cref{lem:distance-y-epsilon-rho} and \eqref{eq:uv-bar-bound} that
	\begin{multline}
		\label{eq:yk-psi-esti}
		\norm{\bar y^{\psi}_{-,k} - \bar y_{-,k}}_{H^1(\Omega)} + \norm{\bar y^{\psi}_{-,k} - \bar y_{-,k}}_{C(\overline\Omega)} \\
		 \leq C_2 \rho_k \left(\norm{u - \bar u}_{L^2(\Omega)} +\norm{v - \bar v}_{L^2(\Gamma)}\right)
	\end{multline}
	and
	\begin{equation}
		\label{eq:yk-nopsi-esti}
		\norm{\bar y_{-,k} - \bar y}_{H^1(\Omega)} + \norm{\bar y_{-,k} - \bar y}_{C(\overline\Omega)} \leq C_2 \epsilon_k,
	\end{equation}
		where the constant $C_2$ is independent of $(\bar u, \bar v)$, $(u, v)$,  $\sigma$, and $k$.
	On the other hand, as a result of \cref{prop:G-pm-belongto-Bouligand-diff} and the definition of $\bar G_{-}$ in \eqref{eq:G-pm-operator}, one has
	\begin{equation}
		\label{eq:G-minus-limit}
		S'(\bar w_{-,k}) \psi \to \bar G_{-}\psi \quad \text{strongly in} \quad H^1(\Omega).
	\end{equation}
	We now have
	\begin{equation} \label{eq:yk-psi-limit}
		\frac{1}{\rho_k}[\bar y^{\psi}_{-,k} - \bar y_{-,k}] = \frac{1}{\rho_k}[\bar y^{\psi}_{-,k} - \bar y_{-,k} - \rho_k S'(\bar w_{-,k})\psi] + S'(\bar w_{-,k})\psi \rightharpoonup e_{-} + \bar G_{-}\psi,
	\end{equation}
	in $H^1(\Omega)$,
	in view of \eqref{eq:eta-k-limit-primal} and \eqref{eq:G-minus-limit}.
	Using \eqref{eq:objective-difference} in \cref{lem:objective-func-difference-G-der}  yields
	\begin{multline*}
		J(\bar w_{-,k} + \rho_k \psi) - J(\bar w_{-,k}) = \frac{1}{2} \norm{\bar y^{\psi}_{-,k} - \bar y_{-,k}}_{L^2(\Omega)}^2 + \frac{\alpha}{2}\norm{\bar y^{\psi}_{-,k} - \bar y_{-,k}}^2_{L^2(\Gamma)}  \\
		\begin{aligned}[b]
			& + \frac{\kappa_\Omega}{2} \rho_{k}^2 \norm{u - \bar u}_{L^2(\Omega)}^2 + \frac{\kappa_\Gamma}{2} \rho_k^2\norm{v - \bar v}_{L^2(\Gamma)}^2 + \scalarprod{\bar y_{-,k} - y_\Omega}{\bar y^{\psi}_{-,k} -\bar y_{-,k}}_{\Omega}  \\
			& + \alpha \scalarprod{\bar y_{-,k} - y_\Gamma}{\bar y^{\psi}_{-,k} -\bar y_{-,k}}_{\Gamma} + \kappa_\Omega \rho_k \scalarprod{\bar u_{-,k}}{u - \bar u}_{\Omega} + \kappa_\Gamma  \rho_k \scalarprod{\bar v_{-,k}}{v - \bar v}_{\Gamma}.
		\end{aligned}
	\end{multline*}	
	From this, \eqref{eq:epsilon-over-rho-k-primal}, \eqref{eq:uk-uv-primal},  \eqref{eq:yk-psi-esti}, \eqref{eq:yk-nopsi-esti}, and \eqref{eq:yk-psi-limit}, one has
	\begin{multline}
		\label{eq:J-k-psi-limit}
		\frac{1}{\rho_k}[J(\bar w_{-,k} + \rho_k \psi) - J(\bar w_{-,k})] \to \scalarprod{\bar y - y_\Omega}{e_{-} + \bar G_{-} \psi}_{\Omega} \\
		\begin{aligned}[t]
			& + \alpha \scalarprod{\bar y - y_\Gamma}{e_{-} + \bar G_{-} \psi}_{\Gamma}  + \kappa_\Omega  \scalarprod{\bar u}{u - \bar u}_{\Omega} + \kappa_\Gamma   \scalarprod{\bar v}{v - \bar v}_{\Gamma}.
		\end{aligned}
	\end{multline}		
	Besides, we deduce from \eqref{eq:objective-difference} in \cref{lem:objective-func-difference-G-der} and \eqref{eq:uk-uv-primal}, as well as \eqref{eq:yk-nopsi-esti}  that
	\begin{equation}
		\label{eq:J-yk-bary}
		|J(\bar w_{-,k}) - J(\bar u, \bar v)| \leq C_{*}(\epsilon_k^2 + \epsilon_k)  \quad \text{for all} \quad k \geq 1.
	\end{equation}
		Here the constant $C_*$ is independent of  $(\bar u, \bar v)$, $(u, v)$,  $\sigma$, and $k$, thanks to  \eqref{eq:uv-bar-bound} and \eqref{eq:y-bar-bound}.
	Furthermore, from \eqref{eq:uk-Uad} and the fact that $u, \bar u \leq u_b$ a.a. in $\Omega$, we have
	\begin{equation*}
		\bar u_{-,k} + \rho_k (u - \bar u) \leq  \bar u+ \epsilon_k(u_b - \bar u) + \rho_k (u_b - \bar u) \leq u_b
	\end{equation*}
	a.a. in $\Omega$ and for $k$ large enough. Similarly, there holds
	$\bar v_{-,k} + \rho_k (v - \bar v) \leq v_b$
	a.a. on $\Gamma$ and for $k$ large enough. We thus obtain
	\[
		\bar w_{-,k} + \rho_k \psi = (\bar u_{-,k},\bar v_{-,k}) + \rho_k (u - \bar u,v - \bar v) \in U_{ad}
	\]
	for sufficient large $k$. From this and the local optimality of $(\bar u, \bar v)$, we arrive at
	\begin{align*}
		0 & \leq \frac{1}{\rho_k}[J(\bar w_{-,k} + \rho_k \psi) - J(\bar u, \bar v)] \\
		& = \frac{1}{\rho_k}[J(\bar w_{-,k} + \rho_k \psi) - J(\bar w_{-,k})] + \frac{1}{\rho_{k}}[J(\bar w_{-,k}) - J(\bar u, \bar v)] \\
		&  \leq \frac{1}{\rho_k}[J(\bar w_{-,k} + \rho_k \psi) - J(\bar w_{-,k})]  + C_{*} \frac{\epsilon_k^2 + \epsilon_k}{\rho_k}
	\end{align*}
	for $k$ large enough, where we have exploited \eqref{eq:J-yk-bary} to derive the last inequality.
	Letting $k \to \infty$ and using the limits in \eqref{eq:epsilon-over-rho-k-primal} and in \eqref{eq:J-k-psi-limit} yields \eqref{eq:primal-OCs-minus}.
\end{proof}



\medskip

We finish this subsection by stating a consequence of \cref{prop:primal-OCs} and \cref{lem:eta-epsilon-estimate}, which will help us to prove \cref{thm:OCs-multiplier}.
\begin{corollary}
	\label{cor:primal-OCs}
	Let $(\bar u, \bar v)$ be a local minimizer of \eqref{eq:P} and $\bar G_{\pm}$ defined  in \eqref{eq:G-pm-operator}. 
	There exists a constant $C>0$ such that, for any $(u,v) \in U_{ad}$, there hold:
	\begin{enumerate}[label=(\alph*)]
		\item \label{item:primal-OCs-minus-not-W} For $z_{-} := \bar G_{-}(u- \bar u, v - \bar v)$,
				\begin{multline}
						\label{eq:primal-OCs-minus-adjoint}
						\scalarprod{\bar p_{-} + \kappa_\Omega \bar u}{u - \bar u}_{\Omega} + \scalarprod{\bar p_{-} + \kappa_\Gamma \bar v}{ v - \bar v }_{\Gamma} 
						\\
						\geq - C |d_1'(\bar t) - d_2'(\bar t)| \norm{\1_{\{ \bar y = \bar t\}} \1_{\{z_{-} \geq 0 \}} z_{-} }_{L^2(\Omega)}
					\end{multline}
				with $\bar p_{-}$  satisfying 
				\begin{equation} \label{eq:adjoint-minus-primal}
						\left \{
						\begin{aligned}
								-\Delta \bar p_{-} + \bar a_{-} \bar p_{-} & =\bar y - y_\Omega && \text{in} \, \Omega, \\
								\frac{\partial \bar p_{-}}{\partial \nuv}  + b(x)\bar p_{-} &= \alpha(\bar y - y_\Gamma)&& \text{on}\, \Gamma;
							\end{aligned}
						\right.
					\end{equation}

		\item \label{item:primal-OCs-plus-not-W} If, in addition, $(\bar u,\bar v) \neq (u_b, v_b)$, then one further has
				\begin{multline}
						\label{eq:primal-OCs-plus-adjoint}
						\scalarprod{\bar p_{+} + \kappa_\Omega \bar u}{u - \bar u}_{\Omega} + \scalarprod{\bar p_{+} + \kappa_\Gamma \bar v}{ v - \bar v }_{\Gamma}
						 \\
						\geq - C |d_1'(\bar t) - d_2'(\bar t)| \norm{\1_{\{ \bar y = \bar t\}} \1_{\{z_{+} \leq 0 \}} z_{+}  }_{L^2(\Omega)},
				\end{multline}
				where $\bar p_{+}$  fulfills 
				\begin{equation} \label{eq:adjoint-plus-primal}
						\left \{
						\begin{aligned}
								-\Delta \bar p_{+} + \bar a_{+} \bar p_{+} & =\bar y - y_\Omega && \text{in} \, \Omega, \\
								\frac{\partial \bar p_{+}}{\partial \nuv}  + b(x)\bar p_{+} &= \alpha(\bar y - y_\Gamma)&& \text{on}\, \Gamma
						\end{aligned}
						\right.
				\end{equation}			
		with  $z_{+} :=\bar G_{+}(u- \bar u, v - \bar v)$.
	\end{enumerate}
\end{corollary}
\begin{proof}
	We prove assertion \ref{item:primal-OCs-minus-not-W} only, since the argument showing \ref{item:primal-OCs-plus-not-W} is similar. By taking $\sigma :=0$ in  \eqref{eq:primal-OCs-minus}, we thus combine the obtained estimate with \eqref{eq:e-minus-esti} in \cref{lem:eta-epsilon-estimate} as well as the Cauchy--Schwarz inequality to derive
			\begin{multline}
					\label{eq:primal-OCs-minus-not-W}
					\scalarprod{\bar y - y_\Omega}{z_{-} }_{\Omega} +\alpha \scalarprod{\bar y - y_\Gamma}{z_{-} }_{\Gamma} + \kappa_\Omega \scalarprod{\bar u}{u- \bar u}_{\Omega} \\
					+ \kappa_\Gamma \scalarprod{\bar v}{v - \bar v}_{\Gamma} \geq - C |d_1'(\bar t) - d_2'(\bar t)| \norm{\1_{\{ \bar y = \bar t\}} \1_{\{z_{-} \geq 0 \}} z_{-} }_{L^2(\Omega)}.
				\end{multline}%
	Since $z_{-} = \bar G_{-}(u- \bar u, v - \bar v)$, there holds
	\begin{equation*} 
		\left \{
		\begin{aligned}
			-\Delta z_{-}+ \bar a_{-} z_{-} & =u - \bar u && \text{in} \, \Omega, \\
			\frac{\partial z_{-}}{\partial \nuv}  + b(x)z_{-}&= v - \bar v&& \text{on}\, \Gamma,
		\end{aligned}
		\right.
	\end{equation*}	
	as a result of the definition of $\bar G_{-}$ in \eqref{eq:G-pm-operator}.
	Testing now the above equation by $\bar p_{-}$ and the equation \eqref{eq:adjoint-minus-primal} by $z_{-}$, we therefore subtract the obtained identities to have
	\[
		\scalarprod{\bar y - y_\Omega}{z_{-} }_{\Omega} +\alpha \scalarprod{\bar y - y_\Gamma}{z_{-} }_{\Gamma} = \scalarprod{\bar p_{-} }{u - \bar u}_{\Omega} +  \scalarprod{\bar p_{-} }{ v - \bar v }_{\Gamma}.
	\]
	Inserting this into \eqref{eq:primal-OCs-minus-not-W} yields \eqref{eq:primal-OCs-minus-adjoint}.
\end{proof}

\subsection{Proofs of \cref{thm:OCs-multiplier,thm:control-ub-vb,thm:strong-B-equivalence}} \label{sec:proofsmain}

In this subsection, we shall prove the main results presented in \cref{sec:OCP} by mainly exploiting \cref{prop:primal-OCs} and \cref{cor:primal-OCs}.

%

\medskip 

\noindent\textbf{Proof of \cref{thm:OCs-multiplier}}.
	We first put $M := \{ \bar y = \bar t\}$ and then divide the proof into two steps as follows.

	\medskip 
	
	\noindent\emph{Step 1: The existence of all desired functions with the subscript containing the minus symbol.}
	We first note for any function $g \in L^2(\Omega)$ that
	\begin{equation}
		\label{eq:distance-positive-cone}
		\norm{\1_{M}g^{+}}_{L^2(\Omega)} = \dist_{\mathcal{C}_M^{-}}(g),
	\end{equation}
	where
	\[
		\mathcal{C}_{M}^{-} := \{ q \in L^2(\Omega) \mid q \leq 0 \quad \text{a.a. in } M \},
	\]
	$\dist_{\mathcal{C}_{M}^{-}}(g)$ stands for 
	the $L^2(\Omega)$-distance 
	from $g$ to $\mathcal{C}_{M}^{-}$, 
	and $g^+$ denotes the positive part of $g$, i.e., $g^+(x) = \max\{g(x),0 \}$ for a.a. $x \in \Omega$.	

	We now define the functional $F: L^2(\Omega) \times L^2(\Gamma) \to (-\infty, + \infty]$ given by
	\begin{multline*}
		F(w) := \scalarprod{\bar p_{-} + \kappa_\Omega \bar u}{u }_{\Omega} + \scalarprod{\bar p_{-} + \kappa_\Gamma \bar v}{ v  }_{\Gamma} \\
		 +C |d_1'(\bar t) - d_2'(\bar t)| \norm{\1_{\{ \bar y = \bar t\}} \1_{\{\bar G_{-}w \geq 0 \}} \bar G_{-}w }_{L^2(\Omega)} + \delta_{U_{ad} - (\bar u, \bar v)}(w),
	\end{multline*}
	for $w := (u,v) \in L^2(\Omega) \times L^2(\Gamma)$,
	where $\bar p_{-}$ is defined in \eqref{eq:adjoint-minus-primal}, $C$ is determined as in \cref{cor:primal-OCs}, and $\delta_{U}$ stands for the indicator function of a set $U$, i.e.,
	\[
		\delta_U(w) = \left\{
		\begin{aligned}
			0 && \text{if} \quad w \in U,\\
			+ \infty && \text{if} \quad w \notin U.
		\end{aligned}
		\right.
	\]
	Obviously, $F$ is convex and lower semicontinuous. 
	Thanks to assertion \ref{item:primal-OCs-minus-not-W} in \cref{cor:primal-OCs}, $F$ attains minimum in $w_* = (0,0)$ and consequently
	\begin{equation}
		\label{eq:F-optimality-condition}
		0 \in \partial F(w_*).
	\end{equation}
	Moreover, the subdifferential of the indicator function $\delta_{U_{ad} - (\bar u, \bar v)}$ at any point $w = (u,v) \in U_{ad} - (\bar u, \bar v)$ coincides with the normal cone to the set $U_{ad} - (\bar u, \bar v)$ at $w$, that is,
	\begin{align*}
		& \partial \delta_{U_{ad} - (\bar u, \bar v)}(w) \\
		= &\Big \{(\zeta_\Omega, \zeta_\Gamma) \in L^2(\Omega) \times L^2(\Gamma) \mid \scalarprod{\zeta_\Omega}{\tilde{u} - u }_\Omega + \scalarprod{\zeta_\Gamma}{\tilde{v} - v }_\Gamma \leq 0 \, \forall (\tilde{u},\tilde{v}) \in U_{ad} - (\bar u, \bar v)  \Big\} \\
		 = & N(U_{ad} - (\bar u, \bar v); w);
	\end{align*}
	see, e.g. \cite{Ioffe,Bauschke2017}.
	A simple computation then gives
	\begin{multline} \label{eq:zeta-multiplier}
		\partial \delta_{U_{ad} - (\bar u, \bar v)}(w_*) = \Big \{(\zeta_\Omega, \zeta_\Gamma)  \in L^2(\Omega) \times L^2(\Gamma) \mid \zeta_\Omega \geq  0 \, \text{ a.a. on } \{\bar u = u_b  \}, \\
		\begin{aligned}[b]
			\zeta_\Omega = 0 \, \text{ otherwise and } \zeta_\Gamma \geq  0 \, \text{ a.a. on } \{\bar v = v_b  \}, \zeta_\Gamma = 0 \, \text{ otherwise}  \Big  \}.
		\end{aligned}
	\end{multline}	
	Define now the functions $F_i: L^2(\Omega) \times L^2(\Gamma) \to \R$, $i=1,2$, as follows
	\begin{align*}
		F_1(w) & := \scalarprod{\bar p_{-} + \kappa_\Omega \bar u}{u }_{\Omega} + \scalarprod{\bar p_{-} + \kappa_\Gamma \bar v}{ v  }_{\Gamma},\\
		\intertext{and}
		F_2(w) & := C |d_1'(\bar t) - d_2'(\bar t)| \norm{\1_{\{ \bar y = \bar t\}} \1_{\{\bar G_{-}w \geq 0 \}} \bar G_{-}w }_{L^2(\Omega)}
	\end{align*}
	for $w:= (u,v) \in L^2(\Omega) \times L^2(\Gamma)$. Easily, $F_1$ is an affine functional and its subdifferential in $w_*$ is defined as
	\begin{equation}
		\label{eq:F1-subdifferential}
		\partial F_1(w_*) =(\bar p_{-} + \kappa_\Omega \bar u, \bar p_{-} + \kappa_\Gamma \bar v).
	\end{equation}
	Moreover, according to the definition of $F_2$ and \eqref{eq:distance-positive-cone}, 
	we have
	\[
		F_2(w) = c_0 \dist_{\mathcal{C}_{M}^{-}}(\bar G_{-}w) \quad \text{with} \quad c_0 := C|d_1'(\bar t) - d_2'(\bar t)|.
	\]
	We then deduce from \cite[Thm.~2, Chap.~4]{Ioffe} (see, also, \cite[Thm.~16.37]{Bauschke2017}) that
	\[
		\partial F_2(w) = c_0 \bar G_{-}^* [\partial \dist_{\mathcal{C}_{M}^{-}}(\bar G_{-}w)].
	\]
	For $w = w_* = (0,0)$, we have $\bar G_{-}w_* = 0$ and thus
	\begin{equation}
		\label{eq:F2-subdifferential}
		\partial F_2(w_*) = c_0  \bar G_{-}^*[N(\mathcal{C}_{M}^{-};0) \cap \bar B_{L^2(\Omega)}];
	\end{equation}
	see, e.g. \cite[Exam.~16.49]{Bauschke2017}, where $\bar B_{L^2(\Omega)}$ stands for the closed unit ball in $L^2(\Omega)$. 
	A simple computation gives
	\begin{equation}
		\label{eq:F2-subdiff-2}
		N(\mathcal{C}_{M}^{-};0) = \{ \tilde{\mu} \in L^2(\Omega) \mid \tilde{\mu} = 0 \, \text{ a.a. in } \Omega \backslash M, \tilde{\mu} \geq 0 \, \text{ a.a. in } M \}.
	\end{equation}
	On the other hand, from the calculus of subdifferential of convex functions (see, e.g., \cite[Thm.~1, Chap.~4]{Ioffe} and \cite[Cor.~16.38]{Bauschke2017}), 
		and from the fact that $\dom F_1 = \dom F_2 = L^2(\Omega) \times L^2(\Gamma)$,	
	we deduce that
	\[
		\partial F(w_*) = \partial F_1 (w_*) + \partial F_2(w_*) + \partial \delta_{U_{ad} - (\bar u, \bar v)}(w_*).
	\]
	Combining this with \eqref{eq:F-optimality-condition}, \eqref{eq:F1-subdifferential}, and \eqref{eq:F2-subdifferential} yields
	\begin{equation}
		\label{eq:multiplier-sum}
		(0,0) = (\bar p_{-} + \kappa_\Omega \bar u, \bar p_{-} + \kappa_\Gamma \bar v) +  C|d_1'(\bar t) - d_2'(\bar t)| \bar G_{-}^*\tilde{\mu}_{-} + (\zeta_\Omega, \zeta_\Gamma)
	\end{equation}
	for some $\tilde{\mu}_{-} \in N(\mathcal{C}_{M}^{-};0)$ and $(\zeta_\Omega, \zeta_\Gamma) \in \partial \delta_{U_{ad} - (\bar u, \bar v)}(w_*)$. 
	Setting $\mu_{-} := C\tilde{\mu}_{-}$ and taking $\tilde{p}_{-}$ as the unique solution to \eqref{eq:adjoint-multiplier-minus} associated with $\mu_{-}$, we have
	\begin{equation}
		\label{eq:stationarity-esti}
		(\bar p_{-}, \gamma(\bar p_{-})) +C|d_1'(\bar t) - d_2'(\bar t)| \bar G_{-}^*\tilde{\mu}_{-} = (\tilde{p}_{-}, \gamma(\tilde{p}_{-})),
	\end{equation}
	as a result of \eqref{eq:adjoint-minus-primal} and \eqref{eq:G-adjoint}. 
	By setting $(\zeta_{-,\Omega}, \zeta_{-,\Gamma}) :=(\zeta_\Omega, \zeta_\Gamma)$, we then conclude from \eqref{eq:stationarity-esti} and \eqref{eq:multiplier-sum} that \eqref{eq:stationary-multiplier-uv-minus} is valid for multipliers with the minus symbol in the subscripts. Moreover, from \eqref{eq:zeta-multiplier}   and \eqref{eq:F2-subdiff-2}, we have \eqref{eq:complementarity-minus} and \eqref{eq:supp-condition-minus}.
	
	\medskip 
	
	\noindent\emph{Step 2: The existence of all desired functions with the subscript containing the plus symbol.} The argument in this step is similar to that in Step 1 with some slight modifications as follows. 
	Corresponding to \eqref{eq:distance-positive-cone}, we observe that
	\begin{equation*}
		\norm{\1_{M}g^{-}}_{L^2(\Omega)} = \dist_{\mathcal{C}_M^{+}}(g),
	\end{equation*}
	where
	\[
		\mathcal{C}_{M}^{+} := \{ q \in L^2(\Omega) \mid q \geq 0 \quad \text{a.a. in } M \}
	\]
	and $g^{-}$ denotes the negative part of $g$, i.e., $g^{-}(x) = \max\{-g(x),0 \}$ for a.a. $x \in \Omega$. 
	Similar to \eqref{eq:F2-subdiff-2}, we have
	\begin{equation*}
		N(\mathcal{C}_{M}^{+};0) = \{ - \tilde{\mu} \in L^2(\Omega) \mid \tilde{\mu} = 0 \, \text{a.a. in } \Omega \backslash M, \tilde{\mu} \geq 0 \, \text{a.a. in } M \}.
	\end{equation*}
	Finally, we now define the functional $\hat F: L^2(\Omega) \times L^2(\Gamma) \to (-\infty, + \infty]$ given by
	\begin{multline*}
		\hat F(w) := \scalarprod{\bar p_{+} + \kappa_\Omega \bar u}{u }_{\Omega} + \scalarprod{\bar p_{+} + \kappa_\Gamma \bar v}{ v  }_{\Gamma} \\
		+C |d_1'(\bar t) - d_2'(\bar t)| \norm{\1_{\{ \bar y = \bar t\}} \1_{\{\bar G_{+}w \leq 0 \}} \bar G_{+}w }_{L^2(\Omega)} + \delta_{U_{ad} - (\bar u, \bar v)}(w),
	\end{multline*}
	for $w := (u,v) \in L^2(\Omega) \times L^2(\Gamma)$, 	where $\bar p_{+}$ is given in \eqref{eq:adjoint-plus-primal} and $C$ is determined as in \cref{cor:primal-OCs}. From this and the same argument as in Step 1, we deduce the desired conclusions.
	\qed

\medskip 

\noindent\textbf{Proof of \cref{thm:control-ub-vb}}.
By 	assertion \ref{item:W-set-minus-zero} in \cref{cor:W-pm-zero}, there holds $W_{-}^\sigma(u_b, v_b; u - u_b,v -v_b) = \{0\}$ for all $(u,v) \in U_{ad}$ and $\sigma \geq 0$. The right-hand side in \eqref{eq:primal-OCs-minus} then  vanishes
	when taking $\sigma := 0$.
Consequently, the right-hand side in \eqref{eq:primal-OCs-minus-adjoint} 
	can be replaced by zero
and we thus have
\[
	\scalarprod{\bar p_{-} + \kappa_\Omega u_b}{u - u_b}_{\Omega} + \scalarprod{\bar p_{-} + \kappa_\Gamma v_b}{ v - v_b }_{\Gamma} \geq 0
\]
for all $u \in L^2(\Omega)$ and $v \in L^2(\Gamma)$ satisfying $u \leq u_b$ a.a. in $\Omega$ and $v \leq v_b$ a.a. in $\Gamma$, where $\bar p_{-}$ satisfies \eqref{eq:adjoint-minus-primal}.
By setting  $\bar p := \bar p_{-}$ and using \eqref{eq:adjoint-minus-primal}, we obtain \eqref{eq:adjoint-state-ubvb}. Moreover, the last variational inequality in combination with a standard argument; see, e.g. the proof of Lemma 2.26 in \cite{Troltzsch2010}, shows \eqref{eq:control-ub-vb-OC}.
\qed

\medskip 

\noindent\textbf{Proof of \cref{thm:strong-B-equivalence}}.
\noindent\emph{Ad \ref{item:strong2B}}. 
Assume that there exist $\tilde{p}$ and $\tilde{a}$ satisfying \eqref{eq:OCs-multipliers-strong} and \eqref{eq:adjoint-multiplier-strong}. 
Without loss of generality, we can assume that 
\begin{equation}
	\label{eq:d1-der-less-d2}
	d_1'(\bar t) < d_2'(\bar t).
\end{equation}
This, in combination with \eqref{eq:OCs-multiplier-sign-condition} and \eqref{eq:Clarke-condition}, implies, respectively, that 
\begin{equation}
	\label{eq:adjoint-sign-levelset}
	\tilde{p} \leq 0 \, \text{a.a. in } \{\bar y = \bar t \} \quad \text{and} \quad \tilde{a}(x) \in [d'_1(\bar t), d'_2(\bar t)] \, \text{for a.a. } x \in \{\bar y = \bar t \} .
\end{equation}
Taking now $(u, v) \in U_{ad}$ arbitrarily, we then deduce from \eqref{eq:J-direc-der} and \eqref{eq:S-dir-der} that
\begin{multline}
	\label{eq:J-dir-der-delta}
	J'(\bar u, \bar v; u- \bar u, v - \bar v) = \scalarprod{\bar y - y_\Omega}{\delta}_\Omega + \alpha \scalarprod{\bar y - y_\Gamma}{\delta}_\Gamma \\
	+ \kappa_\Omega \scalarprod{\bar u}{u - \bar u}_\Omega + \kappa_\Gamma \scalarprod{\bar v}{v - \bar v}_{\Gamma},
\end{multline}
where $\delta$ fulfills
	\begin{equation}
	\label{eq:delta-pde}
	\left \{
	\begin{aligned}
		-\Delta \delta + d'(\bar y;\delta) & = u - \bar u && \text{in} \, \Omega, \\
		\frac{\partial \delta}{\partial \nuv}  + b(x) \delta &= v - \bar v && \text{on}\, \Gamma.
	\end{aligned}
	\right.
\end{equation}
Testing \eqref{eq:adjoint-multiplier-strong} by $\delta$ and \eqref{eq:delta-pde} by $\tilde{p}$, and then subtracting the obtained equations, we have
\begin{multline*}
	\scalarprod{\bar y - y_\Omega}{\delta}_\Omega + \alpha \scalarprod{\bar y - y_\Gamma}{\delta}_\Gamma  
	 - \left[ \scalarprod{\tilde{p}}{u - \bar u}_\Omega + \scalarprod{\tilde{p}}{v - \bar v}_\Gamma \right] \\ = \int_\Omega \tilde{a} \tilde{p} \delta - d'(\bar y; \delta) \tilde{p} dx.
\end{multline*}
By using \eqref{eq:directional-der-d-func} and \eqref{eq:Clarke-condition}, 
a simple computation gives
\begin{align*}
	\int_\Omega \tilde{a} \tilde{p} \delta - d'(\bar y; \delta) \tilde{p} dx & = \int_{\{\bar y = \bar t \} } \tilde{p} \delta [\tilde{a} - d'_1(\bar t) \1_{\{\delta <0\}} - d'_2(\bar t) \1_{\{\delta > 0\}}]dx \\
	& = \int_{\{\bar y = \bar t \} \cap \{ \delta < 0 \} } \tilde{p} \delta [\tilde{a} - d'_1(\bar t)]dx +   \int_{\{\bar y = \bar t \} \cap \{ \delta > 0 \} } \tilde{p} \delta [\tilde{a} - d'_2(\bar t)]dx.
\end{align*}
The last two integrands are nonnegative a.a. in $\{ \bar y = \bar t \}$, thanks to \eqref{eq:d1-der-less-d2} and \eqref{eq:adjoint-sign-levelset}. We then have
\[
	\int_\Omega \tilde{a} \tilde{p} \delta - d'(\bar y; \delta) \tilde{p} dx \geq 0,
\]
which then yields
\begin{equation*}
	\scalarprod{\bar y - y_\Omega}{\delta}_\Omega + \alpha \scalarprod{\bar y - y_\Gamma}{\delta}_\Gamma \geq
	\scalarprod{\tilde{p}}{u - \bar u}_\Omega + \scalarprod{\tilde{p}}{v - \bar v}_\Gamma.
\end{equation*}
From this and \eqref{eq:J-dir-der-delta}, there holds
\begin{align*}
	J'(\bar u, \bar v; u- \bar u, v - \bar v) & \geq  \scalarprod{\tilde{p} + \kappa_\Omega \bar u}{u - \bar u}_\Omega + \scalarprod{\tilde{p} + \kappa_\Gamma \bar v}{v - \bar v}_\Gamma \\
	& = - \scalarprod{\zeta_{ \Omega} }{u - \bar u}_\Omega - \scalarprod{\zeta_{ \Gamma}}{v - \bar v}_\Gamma,
\end{align*}
where we have used \eqref{eq:stationary-multiplier-uv-strong} to derive the last identity. 
Exploiting \eqref{eq:complementarity-strong}, one thus has  \eqref{eq:B-stationarity}.

\medskip 

\noindent\emph{Ad \ref{item:B2strong}}. Assume now that \eqref{eq:B-stationarity} and \eqref{eq:CQ} are fulfilled. 
We first consider the case, where $(\bar u, \bar v) = (u_b, v_b)$. In this situation, by \eqref{eq:CQ}, there holds
\[
	\meas_{\R^2}(\{ \bar y = \bar t \}) =0.
\]
The control-to-state operator is thus differentiable at $(\bar u, \bar v)$, in view of \cref{pro:S-Gateaux-diff-char}. 
From this and a standard argument, we have the desired conclusion.

It remains to consider the case that $(\bar u, \bar v) \neq (u_b, v_b)$.
By \cref{lem:W-pm-sets-characterization}, there holds
\[
	e_{-} := S'(\bar u, \bar v; u - \bar u, v -\bar v) - \bar G_{-} (u - \bar u, v -\bar v) \in W^{0}_{-}(\bar u, \bar v; u - \bar u, v -\bar v).
\]
Setting $z_{-} := \bar G_{-} (u - \bar u, v -\bar v)$, we conclude from \eqref{eq:J-direc-der} and \eqref{eq:B-stationarity} that
\begin{multline*}
	\scalarprod{\bar y - y_\Omega}{z_{-} }_{\Omega} +\alpha \scalarprod{\bar y - y_\Gamma}{z_{-} }_{\Gamma} + \kappa_\Omega \scalarprod{\bar u}{u- \bar u}_{\Omega} 
	+ \kappa_\Gamma \scalarprod{\bar v}{v - \bar v}_{\Gamma} \\
	 \geq - [\scalarprod{\bar y - y_\Omega}{e_{-} }_{\Omega} +\alpha \scalarprod{\bar y - y_\Gamma}{e_{-} }_{\Gamma}],
\end{multline*}%
which, together with \eqref{eq:e-minus-esti} and Cauchy--Schwarz's inequality, gives
\begin{multline*}
	\scalarprod{\bar y - y_\Omega}{z_{-} }_{\Omega} +\alpha \scalarprod{\bar y - y_\Gamma}{z_{-} }_{\Gamma} + \kappa_\Omega \scalarprod{\bar u}{u- \bar u}_{\Omega} \\
	+ \kappa_\Gamma \scalarprod{\bar v}{v - \bar v}_{\Gamma} \geq - C |d_1'(\bar t) - d_2'(\bar t)| \norm{\1_{\{ \bar y = \bar t\}} \1_{\{z_{-} \geq 0 \}} z_{-} }_{L^2(\Omega)}.
\end{multline*}%
This is identical to \eqref{eq:primal-OCs-minus-not-W}. From this and the proof of \cref{cor:primal-OCs}, we have \eqref{eq:primal-OCs-minus-adjoint} and \eqref{eq:adjoint-minus-primal}. 
Similarly, \eqref{eq:primal-OCs-plus-adjoint} and \eqref{eq:adjoint-plus-primal} follow.
By using the same argument in the proof of \cref{thm:OCs-multiplier}, we then have \eqref{eq:OCs-multipliers-minus}--\eqref{eq:adjoint-multiplier-plus}. The proof of \cref{thm:OC-strong} finally implies the existence of $\tilde{p}$ and $\tilde{a}$ satisfying \eqref{eq:OCs-multipliers-strong} and \eqref{eq:adjoint-multiplier-strong}.
\qed

\section{Conclusions} \label{sec:Conclusions}

We have investigated the distributed and boundary optimal control problems for a nonsmooth semilinear elliptic  partial differential equation with unilateral pointwise constraints on both distributed and boundary controls. For any admissible point, by introducing two associated convergent sequences of G\^{a}teaux differentiability admissible controls, we have defined  \emph{left} and \emph{right} Bouligand generalized derivatives of the control-to-state operator at this admissible point.
We have then established the novel optimality condition in terms of these two Bouligand generalized derivatives.
%
In addition, there exist two nonnegative multipliers, which  are associated with the left and right Bouligand generalized derivatives and have supports lying in the level set of the optimal state at the nonsmooth value.  
Under a so-called constraint qualification, this novel optimality condition has been applied in order to derive the corresponding strong stationarity, in which the adjoint state has the same sign on the set of all points at which the nonsmooth coefficient in the state equation is not differentiable. 
	The equivalence between the strong and B- stationarities has also been shown.
	Extension to bilateral constraints is possible by considering a modified lower perturbation function but is left to a follow-up work.	

\section*{Acknowledgements}
	{
	The authors would like to thank Dr. Xuan Thanh Le for his comments, which greatly improve the obtained results of the paper.
	The authors also would like to thank the  anonymous reviewers for their constructive comments which led to notable improvements of the paper.
	}

\section*{Appendix}
\begin{lemma}
	\label{lem:Delta-vanishing-levelset}
	Let $\Omega$ be a bounded domain in $\R^N$, $N \geq 1$, and let $y \in H^1(\Omega)$ be a weak solution to Poisson's equation
	\begin{equation}
		\label{eq:Poisson}
		-\Delta y = f \quad \text{in } \Omega
	\end{equation}
	with a given function $f \in L^2(\Omega)$. Then, for any number $t \in \R$, there holds
	\[
		\Delta y = 0 \quad \text{a.a. in }  \{y = t \}.
	\]
\end{lemma}
\begin{proof}
	It follows from \cite[Thm.~8.8]{GilbargTrudinger2021} (see also \cite[Rem.~(ii), Thm.~1, Sec.~6.3]{Evans2010}) that $-\Delta y = f$ a.a. in $\Omega$, and we thus obtain
	$\Delta y \in L^2(\Omega)$.

	Let us now take a subdomain $\Omega' \subset \overline{\Omega'} \subset \Omega$ arbitrarily. By the interior $H^2$-regularity of weak solutions to \eqref{eq:Poisson} (see \cite[Thm.~8.8]{GilbargTrudinger2021} and \cite[Thm.~1, Sec.~6.3]{Evans2010}), we have
	$y \in H^2(\Omega')$. From this and the behavior of higher-order weak derivatives on level sets \cite[Lem.~4.1]{ChristofMuller2021}, there holds
	\begin{equation}
		\label{eq:vanish-subdomain}
		\Delta y = 0 \quad \text{a.a. in }  \{y = t \} \cap \Omega'.
	\end{equation}
	For any positive integer $n$, we set 
	\[
		\Omega_n := \{ x \in \Omega \mid \dist(x, \Gamma) > \frac{1}{n} \}.
	\]
	Here $\dist(x,\Gamma)$ stands for the distance from $x$ to the boundary $\Gamma$ of the domain $\Omega$.
	We then deduce from \eqref{eq:vanish-subdomain} that
	\[
		\Delta y = 0 \quad \text{a.a. in }  \{y = t \} \cap \Omega_n \quad \text{for all }n \geq 1.
	\]
	We therefore have
	\begin{align*}
		\int_{ \{y = t \}} |\Delta y| \dx & = \int_{ \{y = t \} \cap \Omega_n} |\Delta y| \dx + \int_{ \{y = t \} \cap (\Omega \backslash \Omega_n)} |\Delta y| \dx \\
		& = \int_{ \{y = t \} \cap (\Omega \backslash \Omega_n)} |\Delta y| \dx \\
		& \leq \int_{\Omega \backslash \Omega_n} |\Delta y|    \dx \\
		& = \int_\Omega \1_{\Omega \backslash \Omega_n} |\Delta y|    \dx.
	\end{align*}
	Applying the Cauchy--Schwarz inequality yields
	\[
		\int_{ \{y = t \}} |\Delta y| \dx \leq [\meas_{\R^N}(\Omega \backslash \Omega_n) ]^{1/2}\norm{\Delta y}_{L^2(\Omega)} \to 0 \quad \text{as } n \to \infty.
	\]
	We thus have $\int_{ \{y = t \}} |\Delta y| \dx = 0$ and then obtain the desired identity.
\end{proof}



\bibliographystyle{amsplain}
\bibliography{bouligandocs}
\end{document}